\documentclass[microtype]{gtpart}

\usepackage{pinlabel}
\usepackage{tikz}
\usetikzlibrary{arrows.meta}
\tikzset{>=latex}
\tikzset{inner sep=3pt}
\usepackage[all]{xy}

\title[The Gamma and Strominger--Yau--Zaslow conjectures]{The Gamma and Strominger--Yau--Zaslow conjectures: a tropical approach to periods}

\author[Abouzaid]{Mohammed Abouzaid}
\givenname{Mohammed}
\surname{Abouzaid}
\address{Department of Mathematics, Columbia 
University, New York, NY 10027, USA}
\email{abouzaid@math.columbia.edu}
\urladdr{}

\author[Ganatra]{Sheel Ganatra}
\givenname{Sheel}
\surname{Ganatra}
\address{Department of Mathematics, University of 
Southern California, Los Angeles, CA 90089, USA}
\email{sheel.ganatra@usc.edu}
\urladdr{}

\author[Iritani]{Hiroshi Iritani}
\givenname{Hiroshi}
\surname{Iritani}
\address{Department of Mathematics, Kyoto University, 
Kitashirakawa-Oiwake-cho, Sakyo-ku, Kyoto, 606-8502, Japan}
\email{iritani@math.kyoto-u.ac.jp}
\urladdr{}

\author[Sheridan]{Nick Sheridan}
\givenname{Nick}
\surname{Sheridan}
\address{School of Mathematics, University of Edinburgh, 
Edinburgh EH9 3FD, UK}
\email{nick.sheridan@ed.ac.uk}
\urladdr{}


\arxivreference{1809.02177}

\newcounter{mainthm}
\newtheorem{main}[mainthm]{Theorem}

\newtheorem{thm}{Theorem}[section]
\newtheorem{lem}[thm]{Lemma}
\newtheorem{prop}[thm]{Proposition}

\theoremstyle{remark}
\newtheorem{rmk}[thm]{Remark}
\newtheorem{example}[thm]{Example}

\theoremstyle{plain} 
\newtheorem{conje}[mainthm]{Conjecture}

\def\R{\mathbb{R}}
\def\C{\mathbb{C}}
\def\N{\mathbb{N}}
\def\Z{\mathbb{Z}}
\def\T{\mathbb{T}}

\newcommand{\Log}{\operatorname{Log}}
\newcommand{\Sing}{\operatorname{Sing}}
\newcommand{\vol}{\operatorname{vol}}
\newcommand{\grad}{\operatorname{grad}}
\newcommand{\ch}{\operatorname{ch}}

\newcommand{\wh}[1]{\widehat{#1}}

\newcommand{\bP}{\mathbb{P}}
\newcommand{\Zring}{\mathring{Z}}
\newcommand{\bigO}{O\left(t^\epsilon\right)}
\newcommand{\hG}{\wh{G}}
\newcommand{\hGamma}{\wh{\Gamma}}
\newcommand{\tPhi}{\widetilde{\Phi}}
\newcommand{\iu}{\mathtt{i}}

\newcommand{\cO}{\mathcal{O}}

\def\parfrac#1#2{\frac{\partial #1}{\partial #2}}

\begin{document}

\begin{abstract} 
	We propose a new method to compute asymptotics of periods using tropical geometry, in which the Riemann zeta values appear naturally as error terms in tropicalization. Our method suggests how the Gamma class should arise from the Strominger--Yau--Zaslow conjecture. We use it to give a new proof of (a version of) the Gamma Conjecture for Batyrev pairs of mirror Calabi--Yau hypersurfaces. 
\end{abstract}

\maketitle

\section{Introduction}

\subsection{The `error term' in tropicalization}
\label{subsec:error}

The relationship between tropical and algebraic geometry is based on the `Maslov dequantization':
\[ \log_T\left(T^a + T^b\right) \approx \max(a,b)  \quad \text{ for $T \gg 1$.}\]
Setting $b=0$, we can use this to arrive at the following approximation:
\[ (\log T)^2\int_{-A}^A \log_T\left(1+T^a\right) da \approx (\log T)^2 \int_{-A}^A \max(0,a) da = \frac{A^2 }{2}(\log T)^2.\]
However there is an error term in this approximation (see Figure \ref{fig:Maslov}): in the limit $T \to \infty$ it is given by 
\begin{align*}
(\log T)^2 \int_{-A}^A \left(\log_T \left(1+T^a \right) - \max(0,a)\right) da &= 2 (\log T)^2\int_0^A \log_T\left(1+T^{-a}\right) da \\
&= 2 \int_{T^{-A}}^1 \frac{\log(1+x)}{x} dx \\
&= 2\sum_{k=1}^\infty \frac{(-1)^{k+1}}{k^2} +O(T^{-A}) \\
&= \zeta(2) +O(T^{-A}).
\end{align*}
In other words, $\zeta(2)=\pi^2/6$ arises as a subleading term in the Maslov dequantization. 

Going one dimension up, we can calculate the error term in the analogous approximation 
\[ 
(\log T)^3\int_U \log_T\left(1+T^{a_1} +T^{a_2}\right) da_1 da_2 \approx (\log T)^3 \int_U \max(0,a_1,a_2) da_1da_2.\]
We will assume that $U \subset \R^2$ is a polygon containing the origin, and transverse to the legs of the `tropical curve' $\Sing(\max(0,s_1,s_2))$ (i.e. the locus where at least two of $0,s_1,s_2$ are tied for largest). 
The error term is equal to
\[ \int_{\log T \cdot U} 
\left( \log\left(1+e^{s_1}+e^{s_2}\right) -\max(0,s_1,s_2) \right) ds_1ds_2.
\]
The integrand looks approximately like $\log\left(1+e^s\right)-\max(0,s)$ in the directions normal to the legs of the tropical curve.
Thus the leading piece of the error term is equal to the total length of the tropical curve contained inside the region $\log T \cdot U$ multiplied by $\zeta(2)$, which will be linear in $\log T$. 
It turns out that there is also a constant term, which is equal to $\zeta(3)$ (see Proposition \ref{prop:asymptotics_of_integral}). 

The main idea of this paper is to use such approximations to compute asymptotic expansions for period integrals, and to relate them to the Gamma class of the mirror, which we describe in the next section. 

\begin{rmk} These error terms compute the volume of (parts of) amoebas; see Passare \cite{Passare:zeta2} and Passare and Rullg{\aa}rd \cite{Passare-Rullgard:amoeba} for the study in 2 dimensions. 
\end{rmk}

\begin{figure}[ht]
\centering
\includegraphics[scale=0.4]{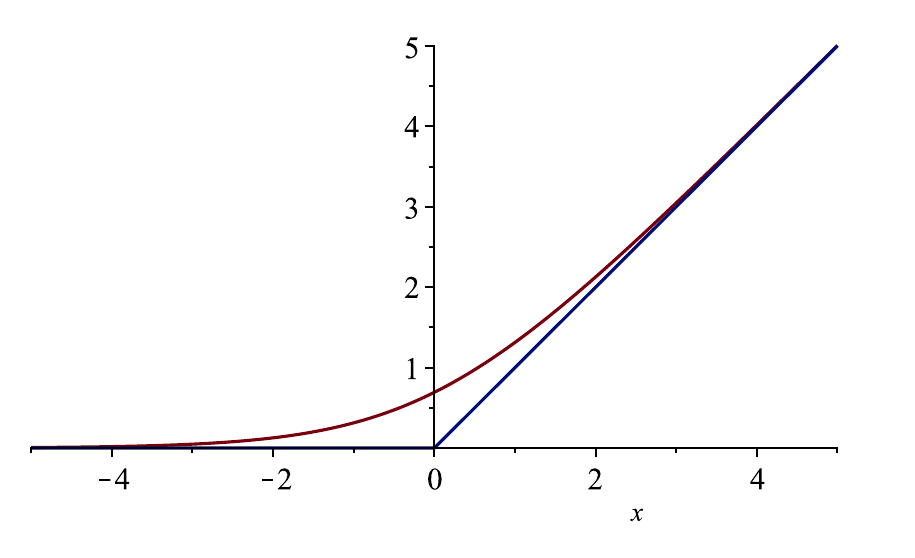}
\includegraphics[scale=0.3]{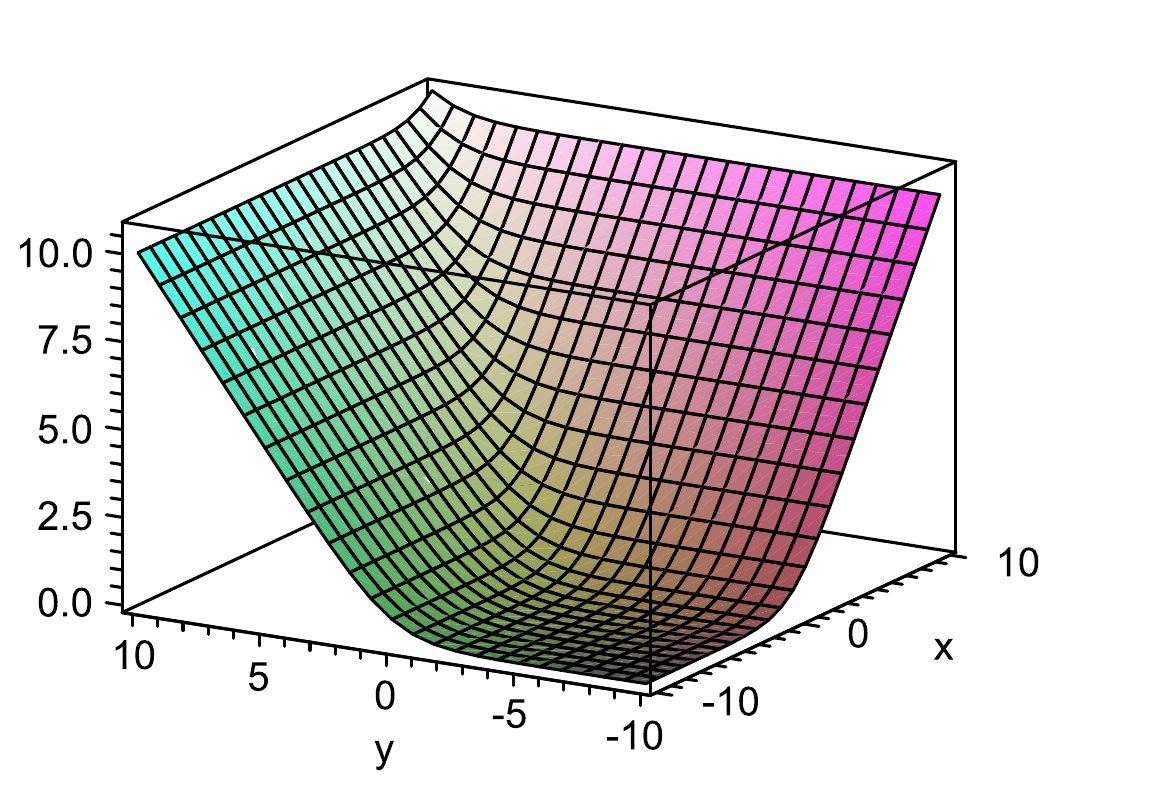}
\caption{The graphs of $\log_T(1+T^x)$ and $\log_T(1+T^x+T^y)$ when  $T=e$. These images were produced using Maple 2018 and Maple 12 respectively \cite{Maple}.}
\label{fig:Maslov}
\end{figure}

\subsection{The Gamma class and mirror periods}

It has been long observed that products of the characteristic numbers of a Calabi--Yau manifold by zeta values can be found in the asymptotics of periods of the  mirror near the large-complex structure limit. For example, $\zeta(3)$ multiplied by the Euler number of a quintic threefold appears in the famous  work of Candelas, de la Ossa, Green and Parkes \cite{CdlOGP:pair}. Later, Hosono, Klemm, Theisen and Yau \cite{HKTY:compint} (see also Hosono, Lian and Yau \cite{Hosono-Lian-Yau:GKZ}) observed that certain Chern numbers of Calabi--Yau complete intersection threefolds can be read off from hypergeometric solutions to the mirror Picard--Fuchs equation. This observation led Libgober \cite{Libgober:Gamma} to introduce the (inverse) Gamma class which makes sense for any almost-complex (or stably complex-oriented) manifold. The \emph{Gamma class}\footnote{When $X$ is an orbifold, the Gamma class has a component in the twisted sector. We nevertheless ignore the twisted sector component since it does not intervene in the statement of the Gamma Conjecture.} of an almost-complex manifold $X$ is defined to be the cohomology class 
\begin{align*} 
\hGamma_X & = \prod_i \Gamma(1+\delta_i) 
= \exp\left(-\gamma c_1(X) + \sum_{k=2}^\infty (-1)^k \zeta(k) (k-1)! \ch_k(TX) \right) \in H^*(X,\R)
\end{align*} 
where $\delta_i$ are the Chern roots of the tangent bundle $TX$ (such that $c(TX) = \prod_i (1+\delta_i)$) and $\gamma = \lim_{n\to \infty} (1+\frac{1}{2} +\cdots+ \frac{1}{n} -\log n)$ is the Euler constant.  
In terms of the Gamma class, a conjecture put forward by Hosono \cite[Conjecture 2.2]{Hosono:central} (see also Horja \cite{Horja:hypergeometric}, Enckevort and van Straten \cite{Enckevort-Straten}, Borisov and Horja \cite{Borisov-Horja:FM},  Almkvist, van Straten and Zudilin \cite{Almkvist-vanStraten-Zudilin:Apery}, Golyshev \cite{Golyshev:deresonating} and Iritani \cite{Iritani:periods}) can be restated as follows: 
\begin{conje}[Gamma Conjecture in the Calabi--Yau case]  
\label{conje:CY_Gamma}
Let $X$ be a Calabi--Yau manifold equipped with a symplectic form $\omega$ and let $\{Z_t\}_{t\in \Delta^*}$ be a family of Calabi--Yau manifolds parametrized by $t$ in a small punctured disc $\Delta^*$ that corresponds to $(X,\omega)$ under mirror symmetry. For a suitable choice of a holomorphic volume form $\Omega_t$ on $Z_t$ and of a coordinate $t$, if a Lagrangian cycle $C_t\subset Z_t$ is mirror to a coherent sheaf $E$ on $X$, then 
\[
\int_{C_t\subset Z_t} \Omega_t = \int_X t^{-\omega} \cdot \hGamma_X \cdot (2\pi \iu)^{\deg/2} \ch(E) + \bigO
\ \text{as $t\to 0$ in a fixed angular sector}
\]
for some $\epsilon>0$, where $\iu=\sqrt{-1}$ is the imaginary unit.
\end{conje} 
\begin{rmk} 
(a) The original conjecture of Hosono \cite{Hosono:central} is stated as an equality between periods and explicit hypergeometric series in the case of complete intersection Calabi--Yau manifolds.  The version presented here can be obtained from the leading asymptotics of the hypergeometric series. 

(b) Both sides of the equality in the Gamma Conjectures are multivalued functions of $t$: on the right hand side, a choice of branch of $\log t$ is required to specify a value for $t^{-\omega}$, while on the left hand side the monodromy of the family $Z_t$ in general acts non-trivially on the homology classes of Lagrangian cycles. The family of Lagrangian cycles $C_t$ mirror to $E$ is identified over the universal cover of the punctured disc $\Delta^*$. 

(c) This is not a mathematically precise conjecture since it depends on mirror symmetry. In the case of Fano manifolds, there is a precise conjecture (the original Gamma Conjecture) which can be formulated purely in terms of quantum cohomology of a Fano manifold $X$ (see Galkin, Golyshev and Iritani \cite{GGI:gammagrass}, Galkin and Iritani \cite{Galkin-Iritani:Gammamirror} and Sanda and Shamoto \cite{Sanda-Shamoto}) and is closely related to Dubrovin's conjecture \cite{Dubrovin:ICM}. 

(d) In the above conjecture, we implicitly assume that $t=0$ is a point of maximal degeneracy in the sense that the associated limit mixed Hodge structure is Hodge--Tate (see Deligne \cite{Deligne:localbehavior}) and that the mirror map takes the form\footnote{If the mirror map is of the form $-\omega \log t + h + O(t)$ with $h\in H^2(X)$, then we need to replace $\hGamma_X$ with $e^h \hGamma_X$ in the conjecture; the class $h$ appears, for instance, when we replace $t$ with $2t$.} $-\omega \log t + O(t)$ so that $t=0$ corresponds to the large-radius limit point of $(X,\omega)$. We could further assume that the volume form $\Omega_t$ is normalized by a Hodge-theoretic condition as discussed in \cite{CdlOGP:pair, Deligne:localbehavior} and Morrison \cite{Morrison:guide}. 

(e) Using the Gamma class, Katzarkov, Kontsevich and Pantev \cite{KKP:Hodge} and Iritani \cite{Iritani:integral} introduced a rational/integral structure on the quantum cohomology, which conjecturally corresponds to the natural rational/integral structure (given by Betti cohomology) on the B-side. 
\end{rmk}

Although Conjecture \ref{conje:CY_Gamma} is not mathematically precise, we can make it precise by specifying what we mean by a ``mirror pair'' and by fixing the correspondence between equivalence classes of cycles on the two sides: 
the Strominger--Yau--Zaslow (SYZ) conjecture posits that mirror pairs should carry dual (possibly singular) torus fibrations, and building on this, the Gross--Siebert program gives a geometric construction of mirror pairs in some large generality (see Strominger, Yau and Zaslow \cite{SYZ} and Gross and Siebert \cite{Gross-Siebert:realaffine}). 
This determines the correspondence between Lagrangian cycles and coherent sheaves appearing in the Gamma Conjecture. The present paper aims to understand/explain the Gamma Conjecture from the viewpoint of the SYZ fibrations. 

\subsection{The Gamma Conjecture for Batyrev mirrors}\label{subsec:gam_bat}

Let $\Delta \subset P_\R$ be a reflexive polytope, and $\nabla \subset Q_\R$ its polar dual, where $P \cong \Z^{n+1}$ and $Q := P^\vee$ and we write $P_K =P\otimes_\Z K$, $Q_K=Q\otimes_\Z K$ for a $\Z$-module $K$. 
Let $V$ be a subset of $\partial\nabla \cap Q$ containing all vertices of $\nabla$ and let $\lambda\colon V \to \R_+$ be a positive real-valued function.\footnote{$\R_+$ denotes the set of positive real numbers.} We assume that there exists a simplicial fan $\Sigma_\lambda$ on $Q_\R$ such that the set of one-dimensional cones of $\Sigma_\lambda$ is  $\{\R_{\ge 0} \cdot q:q\in V\}$ and that $\lambda$ extends to a strictly-convex piecewise-linear function $\lambda\colon Q_\R\to \R$ with respect to the fan $\Sigma_\lambda$. 
We set
\[ f_t(z) := \sum_{q \in V} t^{\lambda_q} \cdot z^q\]
for $t \in \R_+$ and $z\in P_{\C^*}$, and
\[ \Zring_t := \left\{1 = f_t(z) \right\} \subset P_{\C^*}.\]
The positive real locus $C_t^+ \subset \Zring_t$ is defined to be the intersection $\Zring_t \cap P_{\R_+}$; 
this is homeomorphic to a real $n$-dimensional sphere for a sufficiently small $t>0$ 
(see Section \ref{subsec:tropical_setup}).  

Let $Y_\nabla$ denote the toric variety defined by the normal fan of $\nabla$ and take a partial crepant resolution $\wh{Y}_\nabla$ of $Y_\nabla$ which has at worst quotient singularities. The hypersurface $\Zring_t$ compactifies to a quasi-smooth Calabi--Yau hypersurface $Z_t\subset \wh{Y}_\nabla$. 
The holomorphic volume form
\[ 
\Omega_t :=\left. \frac{d\log z_0 \wedge d\log z_1 \wedge \cdots \wedge d\log z_n}{df_t(z)}\right|_{\Zring_t}
\]
also extends to $Z_t$, where $(z_0,z_1,\dots,z_n)$ denotes $\C^*$-coordinates on $P_{\C^*}\cong (\C^*)^{n+1}$.  
On the $B$-side of mirror symmetry we will consider the period integral
\[ 
\int_{C_t^+\subset \Zring_t} \Omega_t.
\]

On the $A$-side of mirror symmetry, we consider the compact convex polytope
\[ \Delta_\lambda := \{ p \in P_\R: \langle q,p\rangle+\lambda_q \ge 0, \ \forall q\in V\}.\]
Our assumption on $\lambda$ ensures that the slopes of the edges at each vertex form a basis of $P_\R$. 
We have a corresponding toric orbifold $Y_{\Delta_\lambda}$ equipped with a K\"ahler class $[\omega_\lambda] = \sum_{q\in V} \lambda_q \cdot D_q$, where $D_q$ is the toric divisor corresponding to the $q$th face $\{p \in \Delta_\lambda: \langle q,p \rangle+\lambda_q =0\}$ of $\Delta_\lambda$. The Batyrev mirror of $Z_t$ is given by a quasi-smooth Calabi--Yau hypersurface $X \subset Y_{\Delta_\lambda}$ (see Batyrev \cite{Batyrev:dual_polyhedra}). It is expected that the large-radius limit of $X$ corresponds to the large complex structure limit $t\to 0$ for $Z_t$ and that the Lagrangian sphere $C_t^+\subset Z_t$ is mirror to the structure sheaf $\cO_X$ of $X$. Thus it makes sense to substitute $C_t=C_t^+$ and $E=\cO_X$ in Conjecture \ref{conje:CY_Gamma}. Our first main result is a proof of this special case of the Gamma Conjecture:

\begin{main}[Gamma Conjecture for the structure sheaf on Batyrev mirror pairs]  
\label{thm:Bat_Gamma_struc}
We have:
\[
\int_{C_t^+\subset Z_t} \Omega_t = \int_X t^{-\omega} \cdot \hGamma_X + \bigO
\quad \text{as $t\to +0$, for some $\epsilon>0$.}
\]
\end{main}

Our second main result is a generalization of Theorem \ref{thm:Bat_Gamma_struc}, in which the structure sheaf $\mathcal{O}_X$ is replaced with an arbitrary ambient line bundle on $X$, i.e., one restricted from $Y_{\Delta_\lambda}$. 
Any line bundle on $Y_{\Delta_\lambda}$ has the form $\mathcal{O}_\nu = \mathcal{O}(-\sum_{q \in V} \nu_q D_q)$ for some $\nu \in \Z^V$; let $\mathcal{L}_\nu$ denote the restriction of $\mathcal{O}_\nu$ to $X$. 
We now describe the cycle $C_t^{(\nu)} \subset Z_t$ mirror to $\mathcal{L}_\nu$. 
Consider the polynomial function $f_{t,\theta}(z)$ on $P_{\C^*}$
\[
f_{t,\theta}(z) := \sum_{q\in V} e^{\iu \theta_q} t^{\lambda_q} z^q 
\]
that we obtain from $f_t(z)$ by multiplying the coefficients $t^{\lambda_q}$ by $e^{\iu \theta_q}$ for some $\theta_q \in \R$, and the associated hypersurface   
\[
\Zring_{t,\theta} := \{f_{t,\theta}(z) = 1\} \subset P_{\C^*}.
\]
Let $C_t^{(\nu)} \subset \Zring_{t}$ be the parallel transport of the positive real cycle $C_t^+ \subset \Zring_t$ as we vary $\theta$ continuously from $\theta=0$ to $\theta = 2\pi \nu$.

\begin{rmk}\label{rmk:why}
Let us explain why the cycle $C_t^{(\nu)}$ is expected to be isotopic to a Lagrangian cycle mirror to $\mathcal{L}_\nu$. 
One expects a commutative diagram of categories as follows:
\[ \xymatrix{ D^bCoh(Y_{\Delta_\lambda}) \ar[d]_-{\text{restriction}} \ar@{<->}[rr]_{\text{HMS}}^\simeq && D^b FS(P_{\C^*},f_t) \ar@{<->}[rrr]_-{\text{compactify}}^\simeq && &D^b FS(\overline{P}_{\C^*},\bar{f}_t) \ar[d]^-{L \mapsto \partial L} \\ D^bCoh(X) \ar@{<->}[rrrrr]_{\text{HMS}}^\simeq &&&&& D^b Fuk(Z_t).}\]
The top arrow labelled `HMS' is homological mirror symmetry for the toric variety $Y_{\Delta_\lambda}$ and its Landau--Ginzburg mirror $(P_{\C^*},f_t)$, where we fix $t$. 
The other arrow on the top line identifies the Fukaya--Seidel category of this Landau--Ginzburg model with that of its fibrewise compactification (see Seidel  \cite{Seidel:II.5}).
The objects of $FS(P_{\C^*},f_t)$ may be taken to be certain Lagrangian submanifolds of $P_{\C^*}$ with boundary on the fibre $f_t^{-1}(1) = \Zring_t$, while the objects of $FS(\overline{P}_{\C^*},\bar{f}_t)$ are Lagrangians with boundary on the compactified fibre $\bar{f}_t^{-1}(1) = Z_t$. 
The left vertical arrow denotes the derived restriction functor, while the right one sends a Lagrangian to its boundary; the commutativity of this diagram appears, e.g., in Auroux \cite[Conjecture 7.7]{Auroux:Tduality}. 
Under the (presumably removable) assumption that $Y_{\Delta_\lambda}$ is smooth, Abouzaid \cite{Abouzaid:homogeneous, Abouzaid:HMS_toric} has constructed certain objects $L_\nu$ of $FS(P_{\C^*},f_t)$, and proved that they are mirror to the line bundles $\mathcal{O}_\nu$  (see also the work of Fang, Liu, Treumann and Zaslow \cite{FLTZ:T-dual}, Fang \cite{Fang:central_charges}, Fang and Zhou \cite{Fang-Zhou}, and Hanlon \cite{Hanlon:monodromy}). 
Therefore, by commutativity of the above diagram, one expects $\mathcal{L}_\nu$ to be mirror to $\partial L_\nu \subset Z_t$. 
One can identify $\partial L_\nu$ with $C_t^{(\nu)}$ up to an isotopy: this becomes transparent using the tropical construction of $C_t^{(\nu)}$ given in Section \ref{sec:phase}. 
\end{rmk}

In light of Remark \ref{rmk:why}, it makes sense to substitute $C_t = C_t^{(\nu)}$ and $E = \mathcal{L}_\nu$ in Conjecture \ref{conje:CY_Gamma}. Our second main result is a proof of this special case of the Gamma Conjecture (generalizing Theorem \ref{thm:Bat_Gamma_struc}, which is the case $\nu=0$):

\begin{main}[Gamma Conjecture for ambient line bundles on Batyrev mirror pairs]  
\label{thm:Bat_Gamma_lb}
We have:
\[
\int_{C_t^{(\nu)} \subset Z_t} \Omega_t = \int_X t^{-\omega}  \cdot \hGamma_X \cdot e^{-\sum_{q \in V} 2\pi \iu \nu_q D_q} + \bigO
\quad \text{as $t\to +0$, for some $\epsilon>0$.} 
\] 
\end{main} 

We remark that similar results have been obtained in \cite[Theorem 1.1]{Iritani:periods}; the novelty in our work is the method of proof, which relates the Gamma Conjecture to the SYZ Conjecture and the Gross--Siebert program. 
In fact, due to the local nature of the computations, we expect that it should not be significantly harder to implement our approach for general Gross--Siebert mirrors than for Batyrev mirrors. 
The Gamma Conjecture for general Gross--Siebert mirrors is open.

In a different direction, we expect that it should be possible to implement our approach to prove the Gamma Conjecture for certain Lagrangian cycles $C_t$ fibring over `tropical cycles' in the base of the SYZ fibration (see Casta\~{n}o and Bernard  \cite{Castano-Bernard2014} and Ruddat and Siebert \cite{Ruddat2019} for the notion of `tropical cycle' in closely-related contexts, and Matessi \cite{Matessi2018a} and Mikhalkin \cite{Mikhalkin2018} for the construction of the corresponding Lagrangian cycles). 
Indeed this is essentially done in Ruddat and Siebert \cite{Ruddat2019}, in the case that the tropical cycle in the base of the SYZ fibration is 1-dimensional. 
In this case the interesting part of the Gamma class (i.e., the part involving zeta values) does not appear in the computation: the mirror coherent sheaf is the skyscraper sheaf of a curve, and in particular its Chern character is concentrated in degrees $\ge 2n-2$, whereas the zeta values in the Gamma class of a Calabi--Yau only appear in degrees $\ge 4$. 
This reflects the fact that the 1-dimensional tropical cycle can be (topologically) deformed to avoid the codimension-2 singular locus of the SYZ fibration, where the non-trivial contributions to the Gamma class are concentrated.\footnote{We should mention that the aim of \cite{Ruddat2019} is rather different from that of the current paper: the authors show that the natural coordinate on the base of the family constructed by Gross--Siebert is a canonical coordinate in the Hodge-theoretic sense.}

\subsection{Proofs of Theorems \ref{thm:Bat_Gamma_struc} and \ref{thm:Bat_Gamma_lb}}

We compute the asymptotics of the period integrals appearing in Theorems \ref{thm:Bat_Gamma_struc} and \ref{thm:Bat_Gamma_lb} by breaking them up into local pieces using tropical geometry. 
This procedure involves an extra layer of combinatorial complexity in the case of Theorem \ref{thm:Bat_Gamma_lb}, so we give the proof of Theorem \ref{thm:Bat_Gamma_struc} first in the name of transparency.

The `local period integrals' that will appear are 
\begin{equation}
\label{eqn:Ilm}
 I_{\ell;m_1,\ldots,m_k} :=  \int_{[0,\infty)^k} s_1^{m_1}\ldots s_k^{m_k} \cdot g_\ell(e^{-s_1},\ldots,e^{-s_k}) \,ds_1\cdots ds_k \quad\text{for $\ell, m_j \in \Z_{\ge 0}$,}
\end{equation}
where
\[g_\ell(X_1,\ldots,X_k) := \sum_{K \subset \{1,\dots,k\}} (-1)^{|K|} \cdot \left(\log \left(1+ \textstyle\sum_{ j \in K} X_j\right)\right)^\ell.\]
The integral \eqref{eqn:Ilm} converges because the integrand decays exponentially at infinity, due to the bound
\begin{equation}
\label{eqn:boundint}
g_\ell(X_1,\ldots,X_k) \le C_\ell \cdot \prod_{j=1}^k X_j \qquad \text{ on $[0,1]^k$.}
\end{equation}
This bound can be proved by observing that the function $g_\ell(X_1,\dots,X_k)$ is analytic 
 in a neighbourhood of $[0,1]^k$, and vanishes along the coordinate hyperplanes $\{X_j = 0\}$, 
 so is divisible by $\prod_{j=1}^k X_j$. We note that $g_\ell=0$ for $\ell = 0$. 

We define a class in $H^*(X)$ by
\begin{equation}
\label{eq:hG_X} 
\hG_X = 1 + \sum_{q,J,\ell,\vec{m}} \frac{I_{\ell;\vec{m}}}{\ell! \prod_{j\in J} m_j!}  \cdot (-D_q) \cdot (-\sigma)^{\ell-1} \cdot \prod_{j \in J} (-D_j)^{m_j + 1}
\end{equation} 
where $\sigma= \sum_{j \in V} D_j$ is the first Chern class of $Y_{\Delta_\lambda}$, and the sum is over all $q \in V$, all nonempty subsets $J \subset V$ not containing $q$, $\ell \ge 1$ and $\vec{m}\in (\Z_{\ge 0})^{J}$. 

\begin{thm}
\label{thm:periodR}
Let $(X, Z_t)$ be a Batyrev mirror pair of Calabi--Yau hypersurfaces and let $C_t^+\subset Z_t$ denote the positive real locus. Then we have
\[ \int_{C_t^+\subset Z_t} \Omega_t = \int_X t^{-\omega_\lambda} \cdot \hG_X + \bigO\quad \text{as $t\to +0$, for some $\epsilon>0$.}\]
\end{thm}

Theorem \ref{thm:periodR} is proved in Section \ref{sec:periodR}. 
The proof uses tropical geometry to decompose $C_t^+$ into pieces, so that the integrals of $\Omega_t$ over these pieces are in one-to-one correspondence with the terms on the right-hand side.  

\begin{thm}
\label{thm:gamma}
We have $\hG_X = \hGamma_X$. 
\end{thm}

Theorem \ref{thm:gamma} is proved in Section \ref{sec:gamma}. 
Combining Theorems \ref{thm:periodR} and \ref{thm:gamma}, we have proved Theorem  \ref{thm:Bat_Gamma_struc}.

In Section \ref{sec:phase}, we generalize Theorem \ref{thm:periodR} to give a computation of period integrals over cycles $C_{t,\theta} \subset \Zring_{t,\theta}$, obtained by parallel transport of $C_t^+$; see Theorem \ref{thm:cycle_with_phase}. 
These include the cycles $C_t^{(\nu)}$, in the special case that all $\theta_q$ are integers. Theorem \ref{thm:Bat_Gamma_lb} is proved by combining Theorems \ref{thm:cycle_with_phase} and \ref{thm:gamma}.

\subsection{Plan}

Theorems \ref{thm:periodR} and \ref{thm:gamma} are proved in Sections \ref{sec:periodR} and  \ref{sec:gamma} respectively; together they prove Theorem \ref{thm:Bat_Gamma_struc}. 
The additional ingredient needed to prove Theorem \ref{thm:Bat_Gamma_lb}, namely Theorem \ref{thm:cycle_with_phase}, is proved in Section \ref{sec:phase}. 
However, the geometric idea underlying our approach may not shine through the tropical combinatorics of the rigorous proofs. 
Therefore, in Section \ref{sec:disc} we explain the idea behind the proofs informally, emphasizing the relationship with the SYZ conjecture and the Gross--Siebert program. 
The reader who has no interest in informal discussions can skip Section \ref{sec:disc}.

\paragraph{Acknowledgements.} 
We thank Denis Auroux for a helpful conversation at an early stage of this project. This work was done during the authors' stay at the Institute for Advanced Study (Fall 2016), Kyoto University (Winter 2017) and the Mathematical Sciences Research Institute (Spring 2018, supported by the National Science Foundation Grant Number DMS-1440140). M.A. was supported  by the National Science Foundation through agreement number DMS-1609148 and DMS-1564172, and by the Simons Foundation
through its ``Homological Mirror Symmetry'' Collaboration grant. S.G. was supported by the National Science Foundation through agreement number DMS-1128155. H.I.~was supported by JSPS KAKENHI Grant Number 16K05127, 16H06335, 16H06337, 17H06127 and 20K03582. N.S.~was partially supported by a Royal Society University Research Fellowship, a Sloan Research Fellowship, and by the National Science Foundation through Grant number DMS-1310604 and under agreement number DMS-1128155.
Any opinions, findings and conclusions or recommendations expressed in this material are those of the authors and do not necessarily reflect the views of the National Science Foundation.
\section{Discussion and examples}
\label{sec:disc}

In this section we sketch the proof of the Gamma Conjecture for the structure sheaf on Batyrev mirrors of dimension at most $3$, emphasizing the relationship with the SYZ conjecture and the Gross--Siebert program. The discussion is not intended to be completely rigorous. 

Observe that the conjecture can be rewritten as
\begin{equation}
\label{eqn:gammai} 
\int_{C_t^+} \Omega_t + \bigO = \sum_{i=0}^n \frac{(-\log t)^{n-i}}{(n-i)!} \cdot \int_X \omega_\lambda^{n-i} \cdot \hGamma_i
\end{equation}
where $\hGamma_i$ is the degree-$2i$ component of $\hGamma_X$. 
 
Roughly speaking we will stratify $C_t^+$ in accordance with the singularities of the SYZ fibration, and we will see that the codimension-$i$ strata give rise to terms in the asymptotic expansion of the period integral which precisely add up to the $i$th term on the right-hand side. 

This is a compelling picture, but unfortunately it becomes more complicated in higher dimensions (compare Remark \ref{rmk:kinks}) and we have not been able to cleanly generalize it.  
The remainder of the paper explains a more pedestrian version of our period computation which works in all dimensions, but which uses the embedding of Batyrev mirror pairs in toric varieties corresponding to dual reflexive polytopes. 

\subsection{Leading term}
\label{sec:leading}

We consider the map $\Log_t \colon P_{\C^*} \to P_\R$ 
\begin{align*}
\Log_t(z_0,\ldots,z_n) = (\log_t|z_0|,\ldots,\log_t|z_n|).
\end{align*}
In the limit $t \to 0$, the \emph{amoeba} $\Log_t(\Zring_t)$ converges to the \emph{tropical amoeba}, which is a codimension-$1$ weighted balanced polyhedral complex (see Mikhalkin \cite{Mikhalkin2004}).
The unique compact component of the complement of the tropical amoeba is precisely the polytope $\Delta_\lambda$ that appears on the $A$-side of our mirror statement.

The image of the cycle $C_t^+$ under $\Log_t$ converges to $\partial \Delta_\lambda$ as $t \to 0$, and the pullback of the volume form $\Omega_t$ to $C_t^+$ converges to the rescaling of the affine volume form on each face by $-\log t$.
Using this we obtain that the leading term of the period integral is
\[ \int_{C_t^+} \Omega_t = (-\log t)^n \cdot \vol\left( \partial \Delta_\lambda \right) + O\left((\log t)^{n-1}\right). \]
The volume on the right-hand side coincides with the sum of symplectic volumes of boundary divisors $D_j \subset Y_{\Delta_\lambda}$ by Guillemin \cite[Theorem 2.10]{Guillemin1994}. 
This coincides with the symplectic volume of $X$ (since $X$ is cohomologous to $\sum_j D_j$), which gives us the leading term in the Gamma Conjecture:
\[ \int_{C_t^+} \Omega_t = \int_X t^{-\omega_\lambda} + O\left((\log t)^{n-1}\right). \]
The sub-leading terms are related to the `bends' in $C_t^+$ where we interpolate between adjacent faces of $\partial \Delta_\lambda$, as we will see in the next sections.

This is closely related to the SYZ conjecture, according to which there should exist a special Lagrangian torus fibration with singularities $Z_t \to B$, where $B\cong \partial \Delta_\lambda$ is endowed with an affine structure.\footnote{Proving the existence of such special Lagrangian torus fibrations remains a difficult question, presenting numerous challenges, see Joyce \cite{Joyce:SYZ}. In practice, our approach only requires the existence of a weak version of an SYZ fibration, similar to that appearing in the Gross--Siebert program. Nevertheless, for the purposes of this informal discussion, we will refer freely to special Lagrangian torus fibrations.}
The cycle $C_t^+$ should correspond to the zero-section $B \subset Z_t$ of this fibration. 
The restriction of the holomorphic volume form $\Omega_t$ to the cycle $C_t^+$  is real, 
and should be approximately equal to the pullback of the affine volume form on $B$ 
(see e.g.~Gross \cite[Sections 1-2]{Gross2013}). 
Thus the leading-order term of the period integral should be
\[ \int_{C_t^+} \Omega_t = (-\log t)^n \cdot \vol(B) + O\left((\log t)^{n-1}\right).\]
The mirror $X$ should admit a dual special Lagrangian torus fibration $X \to B$, and its symplectic volume should coincide with the affine volume of $B$. 
Thus we obtain an explanation of the leading term in the Gamma Conjecture that is similar to the previous one. 
As promised, the codimension-$0$ locus of the base of the SYZ fibration gave rise to the $i=0$ term on the right-hand side of \eqref{eqn:gammai}. 

\begin{rmk} 
The relationship between the leading order asymptotics of periods and tropical geometry has been studied by several people. Mikhalkin and Zharkov \cite{Mikhalkin-Zharkov:trop_Jac} introduced periods for tropical curves in terms of affine length; Iwao \cite{Iwao:int_trop} compared tropical periods for curves with the leading asymptotics of classical periods. Yamamoto \cite{Yamamoto:CY_periods} studied periods (or, radiance obstruction) of tropical K3 hypersurfaces and compared them with classical ones.  
\end{rmk} 

\subsection{K3 surfaces}

Let us consider the two-dimensional case, so $Z_t$ and $X$ are K3 surfaces. 
There should be an SYZ fibration $p:Z_t \to B$ where $B \cong \partial \Delta_\lambda$, compare Gross  \cite{Gross2001,Gross2013} and Ruan \cite{WDRuan:Lag_fibration_survey}. 
We have one affine coordinate chart of $B$ for each face of $\Delta_\lambda$, which has the subspace affine structure; and we also have an affine coordinate chart for each vertex, which is given by projection along the remaining `ray' emanating from the vertex (see Figure \ref{fig:tropK3}). 

\begin{figure}[ht]
\centering 
\begin{tikzpicture}[x=1.2pt, y=1.2pt] 
\draw[very thick] (100,100)--(100,20)--(200,10)--(180,60)--(100,100); 
\draw[very thick] (100,100)--(200,10);

\draw[->,very thick] (100,20)--(80,0); 
\draw[very thick] (200,10)--(220,-5);
\draw[very thick] (180,60)--(200,70);
\draw[->, very thick] (100,100)-- (90,130); 

\draw[very thin] (100,20)--(180,60);  
\draw[very thin] (120,40)--(100,100); 
\draw[very thin,->] (120,40)--(90,100);
\draw[very thin,->] (120,40)--(88,72);
\draw[very thin,->] (120,40)--(85,40); 
\draw[very thin] (120,40)--(100,20);

\filldraw (100,40) circle [radius=1.5]; 
\filldraw (100,60) circle [radius=1.5]; 
\filldraw (100,80) circle [radius=1.5]; 

\filldraw (125,17.5)  circle [radius=1.5]; 
\filldraw (150,15) circle [radius=1.5]; 
\filldraw (175,12.5) circle [radius=1.5]; 

\filldraw (125,77.5)  circle [radius=1.5]; 
\filldraw (150,55) circle [radius=1.5]; 
\filldraw (175,32.5) circle [radius=1.5]; 

\filldraw (195,22.5) circle [radius=1.5]; 
\filldraw(190,35) circle [radius=1.5]; 
\filldraw(185,47.5) circle [radius=1.5]; 

\filldraw(120,90) circle [radius=1.5]; 
\filldraw(140,80) circle [radius=1.5]; 
\filldraw(160,70) circle [radius=1.5]; 

\filldraw (100,100) circle [radius=1.5]; 
\filldraw (100,20) circle [radius=1.5];
\filldraw (200,10) circle [radius=1.5];
\filldraw (180,60) circle [radius=1.5];

\fill[yellow, opacity=0.7] (109,29) -- (171,23) -- (109,79) -- (110,30); 

\shade[top color=blue, bottom color=blue!30,opacity=0.6] (100,100) -- (100,93) arc[x radius=7,y radius =7, start angle =-90, end angle =-24]--cycle; 
\shade[left color=blue,right color=blue!30,opacity=0.6] (100,75) arc[x radius =5, y radius =5, start angle=-90, end angle =90] -- cycle; 
\shade[left color=blue,right color=blue!30,opacity=0.6] (100,55) arc[x radius =5, y radius =5, start angle=-90, end angle =90] -- cycle; 
\shade[left color=blue,right color=blue!30,opacity=0.6] (100,35) arc[x radius =5, y radius =5, start angle=-90, end angle =90] -- cycle; 
\shade[bottom color=blue, top color=blue!30,opacity=0.6] (100,27) arc[x radius =7, y radius =7, start angle=90, end angle =-10] -- (100,20) -- cycle; 

\fill[red,opacity=0.1] (100,44) -- (104,46) -- (104,55) -- (100,56) -- cycle; 
\draw[red,very thick,opacity=0.4] (100,44) -- (104,46) -- (104,55) -- (100,56);
\fill[red,opacity=0.2] (100,44) -- (106,44) -- (106,51) -- (100,56) --cycle; 
\draw[red, very thick] (100,44) -- (106,44) -- (106,51) -- (100,56); 

\end{tikzpicture} 
\caption{The tropical amoeba of a mirror quartic $\Zring_t = \{ tW_1+tW_2+tW_3+t/(W_1W_2W_3)=1\}$}
\label{fig:tropK3}
\end{figure}
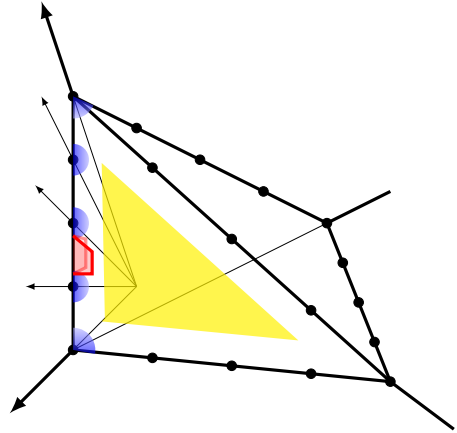 

The resulting affine structure on $B$ is defined everywhere except near certain points along the edges of $\Delta_\lambda$, which correspond to the intersections of $Z_t$ with codimension-$2$ toric strata of $Y_\nabla$. 
Generically, there are 24 of these, so we end up with an affine structure on the 2-sphere with 24 singularities. 

Away from a neighbourhood of the singularities, the holomorphic volume form $\Omega_t$ is approximately equal to the flat volume form to order $\bigO$, so 
\[ \int_{p^{-1}(U)\cap C_t^+} \Omega_t = (-\log t)^2 \cdot \vol(U) + \bigO.\]
In a neighbourhood of a singularity, if we throw out terms of order $\bigO$ then the local model for $(Z_t,C_t^+,\Omega_t)$ is
\begin{align*}
Z_t &= \{(y_1,y_2,x) \in \C^2 \times \C^* : y_1 y_2=1+x\},\\
C_t^+ &= Z_t \cap (\R_+)^3,\\
\Omega_t &= \frac{dy_1 \wedge dy_2}{x} =  d\log y_1 \wedge  d\log x
\end{align*}
where $\{y_1y_2 = 0\}$ corresponds to the boundary divisor of $Y_\nabla$ (compare Kontsevich and Soibelman \cite{Kontsevich-Soibelman:affine}). 

\begin{example} 
Let $X \subset \C\bP^3$ be a quartic K3 surface equipped with a  symplectic form $\omega$ in the class $c_1(\C\bP^3)$. The mirror is given by $\Zring_t=\{tW_1+tW_2+tW_3+t/(W_1W_2W_3)=1\}$. The tropical amoeba of $\Zring_t$ is shown in Figure \ref{fig:tropK3}; the $\Log_t$-image of the positive real cycle $C_t^+$ converges to the boundary of the simplex $\Delta_\lambda= \{w_1\ge -1, w_2\ge -1, w_3\ge -1, w_1+w_2+w_3\le 1\}$ as $t\to +0$, where $w_i= \log_t|W_i|$. We cover $\partial \Delta_\lambda$ by affine charts: on the interior of a facet (yellow region), we consider the subspace affine structure, and around a lattice point $v$ on an edge (blue region), we consider the affine structure given by the projection along the ray $\R_+ v$. The singularities of the affine structure occur somewhere between adjacent lattice points on edges. For example, consider the red region in Figure \ref{fig:tropK3}, which lies between the two affine charts $(w_2-w_1, w_2+w_3+1)$, $(w_2-w_1,w_3)$ associated with the rays $\R_+(-1,-1,1)$ and $\R_+(-1,-1,0)$. Since $tW_3=t^{1+w_3}$ and $t/(W_1W_2W_3) = t^{1-w_1-w_2-w_3}$ are exponentially small near the red region, the cycle $C_t^+$ in this region is given by the equation 
\[
tW_1 + t W_2 \approx 1  \quad 
\Longleftrightarrow \quad 
t^{-(w_2-w_1)} + 1 \approx t^{-w_2-1} = t^{-w_2-w_3-1}\cdot  t^{w_3}. 
\]
Setting $x=t^{-(w_2-w_1)}$, $y_1 = t^{-(w_2+w_3+1)}$, 
$y_2=t^{w_3}$, we find that the red region of the cycle is approximated by the positive real locus of the local model $1+x = y_1y_2$ above. Note that $(\log_t x, \log_t y_1)$ and $(\log_t x, \log_t y_2)$ give affine charts of the adjacent blue regions. 
\end{example} 

The base of the SYZ fibration at such a point is a `focus-focus singularity'. 
An approximation to the SYZ fibration $p:Z_t \to B$ can be written down away from the region where both $y_i$ are small. 
The approximation is defined using coordinates $b=-\log_t |x|$, $c_1=-\log_t |y_1|$, $c_2=-\log_t|y_2|$; we have $p \approx (b,c_1)$ away from $y_1=0$ and $p\approx (b,c_2)$ away from $y_2=0$.    
We observe that 
\[
c_1 + c_2 =- \log_t|y_1y_2| = -\log_t|1+x| \approx \max(0,b)
\] 
when $b \gg 0$ or $b \ll 0$, so the transition maps for the approximate SYZ fibration are approximately affine-linear in these regions. 
The fact that these transition maps are different for large and small $b$ accounts for the non-trivial monodromy of the affine structure around the focus-focus singularity. 

\begin{figure}[ht]
\centering
\begin{tikzpicture}
\node at (0,0) {\includegraphics[scale=0.45]{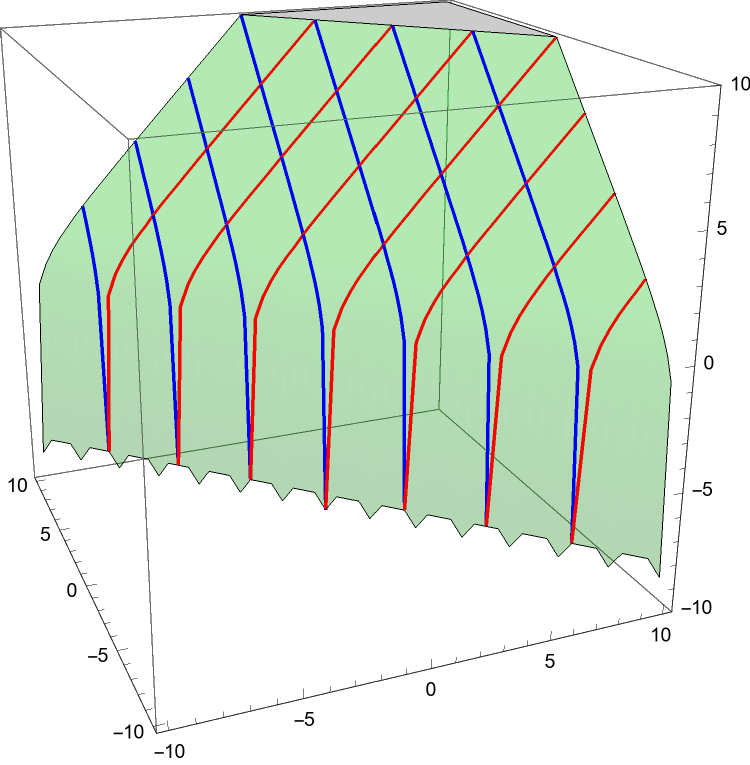}};
\draw[blue,->] (3,1.5) -- (1.7,1.5) -- (1.29,1.22);
\filldraw[blue] (3,1.5) node [fill=white,draw] {\scriptsize $c_1=\text{const}$}; 
\draw[red,->] (-2,2.4)--(-1.52,1.55); 
\filldraw[red] (-2.5,2.6) node [fill=white, draw] {\scriptsize $c_2=\text{const}$}; 
\draw[->] (-2.1,-3.25)--(-0.8,-2.93); 
\draw (-0.7,-3.1) node {\footnotesize $c_1$}; 
\draw[->] (-2.1,-3.25)--(-2.2,-1.25); 
\draw (-2,-1.4) node {\footnotesize $b$}; 
\draw[->] (-2.1,-3.25)--(-2.83,-1.75);
\draw (-3,-1.9) node {\footnotesize $c_2$}; 

\draw[->] (5,0)--(10,0);
\draw[->] (7.5,-3)--(7.5,3); 
\filldraw[opacity=0.3] (6,2)--(6,-2)--(7.5,-2)--(9,-0.5)--(9,2)--cycle;
\draw (10.2,0) node {$b$}; 
\draw (7.5,3.2) node {$c_1$};
\draw (5.8,-0.2) node {$-B$}; 
\draw (9,-0.2) node {$B$}; 
\draw (7.3,2) node {$C_1$}; 
\draw (7.15,-2.07) node {$-C_2$}; 

\filldraw (7.5,0) node {$\times$} ;  

\draw[blue] (5,1) node {$c_1=\text{const}$}; 
\draw[blue] (6,1.6)--(9,1.6);
\draw[blue] (6,1.2)--(9,1.2); 
\draw[blue] (6,0.8)--(9,0.8); 
\draw[blue] (6,0.4)--(9,0.4);  

\draw[red] (5,-1) node {$c_2=\text{const}$}; 
\draw[red] (6,-0.4)--(7.5,-0.4)--(9,1.1); 
\draw[red] (6,-0.8)--(7.5,-0.8)--(9,0.7); 
\draw[red] (6,-1.2)--(7.5,-1.2)--(9,0.3); 
\draw[red] (6,-1.6)--(7.5,-1.6)--(9,-.1); 

\end{tikzpicture} 
\caption{The left picture: the surface  
  $c_1+c_2=-\log_t(1+t^{-b})$ with $t=1/e$. The red and blue coordinate lines show how the approximately-affine charts $(b,c_1)$, $(b,c_2)$ are glued. (This image was produced using the Wolfram Cloud \cite{Wolfram}.) The right picture: the affine manifold with singularity drawn on the $(b,c_1)$-plane (which corresponds to the red region in Figure \ref{fig:tropK3}); this arises from the left picture in the limit $t\to 0$.}

\label{fig:2dim_sing}
\end{figure}
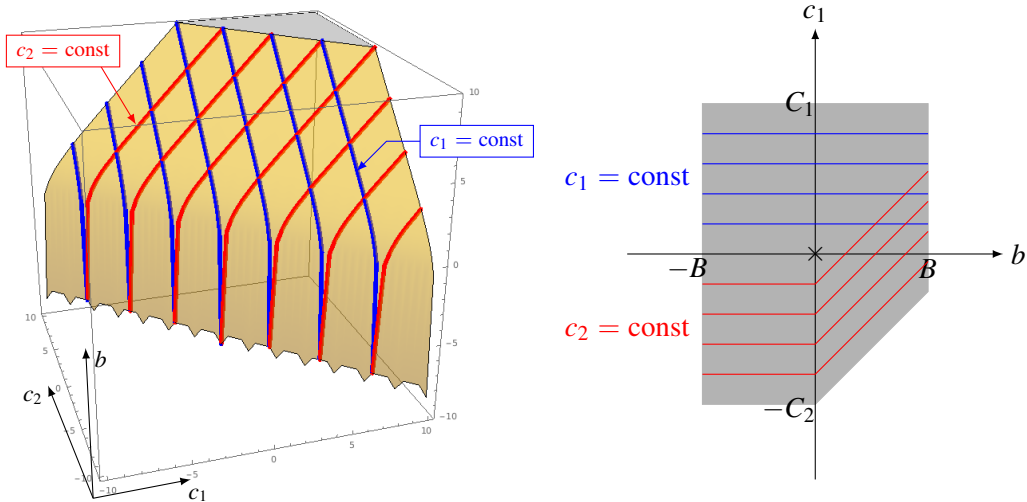

Figure \ref{fig:2dim_sing} shows the hypersurface $C_t^+ = \left\{c_1 + c_2 = -\log_t\left(1+t^{-b}\right)\right\} \subset \R^3$, with the horizontal coordinates corresponding to $c_1$ and $c_2$, and the vertical coordinate to $b$. 
On it we draw the level sets of the coordinates of the SYZ fibration, where each is defined. 
We cut out a region $p^{-1}(U)\cap C_t^+$ where $U$ is a neighbourhood of the singularity in the base of the SYZ fibration. 
It will have boundaries $-B \le b \le B$ for some large $B$; $c_1 \le C_1$ for some large $C_1$, so that $c_1$ is a coordinate of the approximate SYZ fibration along that boundary; and $c_2 \le C_2$ for some large $C_2$, for the same reason. We assume $B\ll C_1+C_2$ so that the boundaries $c_1=C_1$ and $c_2=C_2$ do not intersect. 
We integrate $\Omega_t = (-\log t)^2 db \wedge dc_1$ over $p^{-1}(U) \cap C_t^+$, which means we calculate the area of its projection to the $b$-$c_1$ plane.
This is the area of the region $\{(b,c_1):-B \le b \le B, -\log_t(1+t^{-b}) - C_2 \le c_1 \le C_1\}$, which is clearly
\[
(-\log t)^2 \int_{-B}^B \left( C_1+C_2+\log_t(1+t^{-b})\right) db.
\]
In contrast, the affine volume of $U$ will be
\[ 
(-\log t)^2 \int_{-B}^B \left( C_1+C_2 -\max(0,b) \right) db.
\]
The difference between these two is the contribution of this region to the sub-leading terms of our period integral. 
It is equal to
\begin{align*}
(-\log t)^2 \int_{-B}^B \left( \max(0,b) + \log_t(1+t^{-b}) \right) db 
= - \zeta(2) + O(t^B) 
\end{align*}
as we observed in the Introduction (see Section \ref{subsec:error}).  
Thus we have established that each of the $24$ singular points in the SYZ base (i.e., the codimension-$2$ strata) gives rise to a contribution of $-\zeta(2)$ to the sub-leading term in the period integral. 
These terms sum to
\[-24 \zeta(2) = \int_X \hGamma_2,\]
using the fact that $\hGamma_2 = -c_2(TX) =-24[{\rm pt}]$ for a $K3$ surface, which is the $i=2$ term in the right-hand side of \eqref{eqn:gammai} as promised. 
This completes the sketch proof of Theorem \ref{thm:gamma} in dimension $2$.

The complete proof of Theorem \ref{thm:gamma} that we give in Section \ref{sec:gamma} applies even in situations where $X$ is not smooth but only quasi-smooth, which means (in this two-dimensional case) that some of the singular points in the SYZ base have collided. 
There is a new phenomenon here, which we briefly indicate without going into full details. 

We consider the tropical polynomial 
\[ f_a(b) := \max(-b,a,b),\]
and the leading behaviour of the corresponding `error in tropicalization' integral
\[ I(a,B,t) := (-\log t)^2 \int_{-B}^B \left(f_a(b)+ \log_t \left(t^{b}+t^{-a}+t^{-b} \right) \right) \, db\]
as $t \to +0$, with $a,B$ held fixed and satisfying $|a|\ll B$. 
When $a > 0$, $f_a(b)$ has two bends and is `tropically smooth' at both (i.e., the slope changes by $1$). 
However when $a \le 0$, we have $f_a(b) = \max(-b,b)$ and the two bends have collided into a single bend which is not tropically smooth (the slope changes by $2$). 

This is reflected in the behaviour of the integral: when $a>0$, the two bends in $f_a$ each contribute $-\zeta(2)$ to the leading term of the integral, by the computation of Section \ref{subsec:error}, so $I(a,B,t) = -2\zeta(2)+O(t^\epsilon)$ for some $\epsilon>0$. 
When $a<0$ the terms involving $a$ contribute negligibly to the integral; after dropping these terms, a straightforward manipulation reduces the computation to that of Section \ref{subsec:error}, giving the answer $-\zeta(2)/2 + O(t^\epsilon)$. 
This reflects the fact that, although two separate focus-focus singularities each contribute $-\zeta(2)$ to the period integral, after they collide the contribution is only $-\zeta(2)/2$.

The corresponding local model for $Z_t$ is given by $y_1y_2 = x+t^{-a}+x^{-1}$. The discontinuity of the constant term in the asymptotics of periods can be understood from the fact that the large-complex structure limit of $Z_t$ is different between $a>0$ and $a<0$. 

This collision of two focus-focus singularities in the SYZ base of $Z_t$ is mirror to a degeneration of $X$ so that it acquires an $A_1$ singularity. 
Indeed, a local picture for the development of this $A_1$ singularity is given by the family of toric varieties with moment polytopes $\{(b,c): c \ge f_a(b)\}$ as $a$ passes from positive to negative. 
We consider the effect of this degeneration on the $i=2$ term in the right-hand side of \eqref{eqn:gammai}, which is
\[ \int_X \hGamma_2 = -\zeta(2) \cdot \chi(X).\]
For $a>0$, the local contribution to the Euler characteristic is $2$, from the two toric fixed points; for $a \le 0$ the local contribution is $1/2$, from the single toric fixed point which is an orbifold point of order $2$. 

Thus the effect of the collision of two focus-focus singularities on the period integral, and on the mirror integral \eqref{eqn:gammai}, is the same: $-2\zeta(2)$ gets replaced by $-\zeta(2)/2$. 
A similar phenomenon can be observed with the collision of $k$ focus-focus singularities, replacing $-k\zeta(2)$ with $-\zeta(2)/k$. 

\begin{rmk} 
The asymptotics becomes subtle when $a=0$. 
More generally, we can consider the local model $Z_t$ defined by 
$y_1y_2 = x + c t^{-a} + x^{-1}$. 
For non-zero $a$, the corresponding `error in tropicalization' integral 
does not depend on the coefficient $c$. 
For $a=0$, however, it depends analytically on $c$ as follows: 
\begin{multline*} 
(-\log t)^2 \int_{-\infty}^\infty \left( 
\max(-b,b) + \log_t(t^b+c+t^{-b})\right) db 
 = -2 \int_0^1 \log(x^2+ c x + 1)  \frac{dx}{x} \\
 = \arcsin^2\left(\frac{c}{2}\right) - \pi 
\arcsin\left(\frac{c}{2}\right) - \frac{\pi^2}{12}.  
\end{multline*}  
It is interesting to note that this gives $-2 \zeta(2)$ for $c=2$. 
The point $(a,c)=(0,2)$ is the so-called conifold point 
in the complex moduli space, and should be mirror to a smoothing $T^*S^2$ 
of the $A_1$-singularity. 
The value $-2\zeta(2)$ can be interpreted as $-\chi(T^*S^2)\zeta(2)$. 
\end{rmk}

\subsection{Threefolds}

Now we consider the case where $Z_t$ is 3-dimensional. 
In this case there is again an SYZ fibration $p: Z_t \to B$ with $B \cong \partial \Delta_\lambda$, but the singular locus is more complicated: it generically consists of a trivalent graph lying inside the codimension-$1$ locus of $\partial \Delta_\lambda$ \cite{Gross2001,Gross2013,WDRuan:Lag_fibration_survey}, and there are two types of vertices:  
those lying in the interior of a codimension-$1$ face, with the three incident edges all lying in the same face (which we will call `type I'); and those lying at the intersection of three codimension-$1$ faces, with the three incident edges all lying in different faces (which we will call `type II'). 

\subsubsection{The edges} 
\label{sec:edges}

Along an edge of the singular locus, the SYZ fibration is a product of the two-dimensional case previously considered with an $S^1$-fibration over an interval. 
Thus the integral along the edges should contribute $-\zeta(2)\cdot (-\log t) \cdot (\text{total length of edges})$, which comes out equal to 
\[ -\zeta(2) \cdot (-\log t) \cdot \int_X \omega_\lambda \cdot c_2(TX) = \int_X t^{-\omega_\lambda} \cdot \hGamma_2\]
as required ($c_2(TX)$ is represented by the singular locus of the fibration; see Gross  \cite[Theorem 2.17]{Gross2001}). 

\subsubsection{Type I vertex: $y_1y_2y_3=1+x$}
\label{sec:typei}

The local model near a type I vertex is 
\begin{align*}
Z_t &= \{(y_1,y_2,y_3,x)\in \C^3\times \C^*: y_1 y_2 y_3 = 1+x\}\\
C_t^+ &= Z_t \cap (\R_+)^4\\
\Omega_t &= \frac{dy_1 \wedge dy_2 \wedge dy_3}{x} =  d\log x \wedge d\log y_1 \wedge d\log y_2,
\end{align*}
where the boundary divisor of $Y_\nabla$ corresponds to $\{y_1y_2y_3=0\}$.

The SYZ fibration is approximated using coordinates $b=-\log_t|x|$ and $c_i = -\log_t |y_i|$ as before.  
We set $p \approx (b,c_2,c_3)$ away from $y_1 = 0$, $p \approx (b,c_1,c_3)$ away from $y_2=0$, and $p \approx (b,c_1,c_2)$ away from $y_3=0$. 
Observe that $c_1+c_2+c_3 \approx \max(0,b)$ away from $b \approx 0$ as before, so once again the transition maps are affine-linear away from this area. 
The region $p^{-1}(U)\cap C_t^+$ will be cut out by inequalities $-B \le b \le B$, $c_i \le C_i$ as before, and we must calculate its projection to the $b$-$c_1$-$c_2$-plane. 
The projection to the $b$-$c_1$-$c_2$ is cut out by inequalities
\[
-B \le b\le B,\quad c_1 \le C_1, \quad c_2 \le C_2, \quad c_1+c_2 \ge -\log_t(1+t^{-b}) - C_3.
\]
We assume $B\ll C_1+C_2+C_3$ to ensure that the fibre of this region over $b\in [-B,B]$ is nonempty. 
The fibre of this region over $b \in [-B,B]$ is a right-angle isosceles triangle whose sidelengths are easily calculated, which gives the total volume of the region as
\[ 
(-\log t)^3\int_{-B}^B \frac{\left(C_1+C_2+C_3 +  \log_t(1+t^{-b})\right)^2}{2} db.
\]

As before, we need to subtract off the affine volume of the region, which is 
\[ 
(-\log t)^3 \int_{-B}^B \frac{\left(C_1+C_2+C_3 - \max(0,b)\right)^2}{2} db.
\]
However even after subtracting off this volume, we will still get a divergent integral as $t$ goes to $+0$. 
That is because of the contributions from the edges of the discriminant locus: the three edges each contribute a term 
\[ -(-\log t)\cdot C_i \cdot  \zeta(2) \approx (-\log t)^3 \cdot C_i \cdot \int_{-B}^B \left(\max(0,b) +\log_t(1+t^{-b}) \right) db
\]
to the integral. 
When we subtract off these contributions from the legs, we end up with the contribution which arises solely from the vertex of the discriminant locus, which is given by the integral
\[ 
(-\log t)^3 \int_{-B}^B \frac{\left(-\log_t(1+t^{-b})\right)^2 - \max(0,b)^2}{2} db = \zeta(3) + O(t^\epsilon). 
\]
We shall prove this later, see \eqref{eq:zeta_n}.  

\subsubsection{Type II vertex: $y_1y_2 = 1+x_1+x_2$}
\label{sec:typeii}

The local model near a type II vertex is
\begin{align*}
Z_t &= \{(y_1,y_2,x_1,x_2)\in \C^2\times (\C^*)^2:y_1y_2 = 1+ x_1+x_2\} \\
C_t^+ &= Z_t \cap (\R_+)^4\\
\Omega_t &= \frac{dy_1 \wedge dy_2 \wedge dx_1}{x_1x_2} =  d\log x_1 \wedge d\log x_2 \wedge d\log y_1,
\end{align*}
where the boundary divisor of $Y_\nabla$ corresponds to $\{y_1y_2=0\}$.

The SYZ fibration is approximated using coordinates $b_i =-\log_t |x_i|$ and $c_i = -\log_t |y_i|$. 
The region $p^{-1}(U)\cap C_t^+$ will be cut out by $(b_1,b_2) \in V$ for some region $V \subset \R^2$ enclosing the origin, together with $c_i \le C_i$, and we must calculate its projection to the $b_1$-$b_2$-$c_1$-plane. We assume that $\max(0,b_1,b_2)\ll C_1+C_2$ for $(b_1,b_2)\in V$. 
We find that this area is equal to 
\[ 
(-\log t)^3 \int_V \left(C_1+C_2 + \log_t(1+t^{-b_1}+t^{-b_2})\right) db_1 db_2.
\]
Once again we subtract off the affine volume of the region, leaving
\[ 
(-\log t)^3\int_V \left( \max(0,b_1,b_2) + \log_t(1+t^{-b_1} + t^{-b_2}) \right) db_1 db_2.
\]
Next we need to subtract off the sum of the contributions from the edges of the discriminant locus, which is equal to $\zeta(2)$ multiplied by the total length $L$ of the standard tropical line $\Sing(\max(0,b_1,b_2))$ contained in the region $V$. 
We shall show in Proposition \ref{prop:asymptotics_of_integral} that  
\begin{align}
\label{eqn:typeII}
\begin{split}  
\lim_{t\to +0} \left[  (-\log t)^3 \int_{V} \right.  &\left(\max(0,b_1,b_2) + \log_t(1+t^{-b_1} + t^{-b_2}) \right)db_1 db_2  \\
 & \quad + \zeta(2)\cdot (-\log t) \cdot L \bigg] = -\zeta(3)
\end{split}
\end{align}
so the contribution of a Type II vertex in the discriminant locus to the overall integral is $-\zeta(3)$.

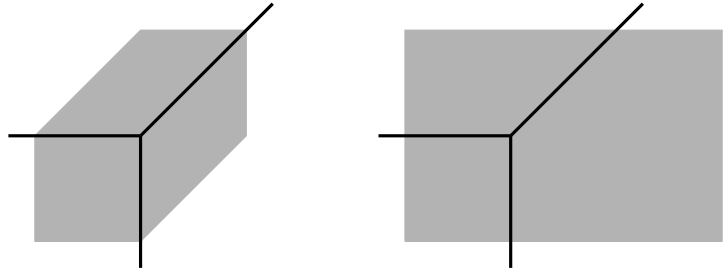
\begin{figure}[ht]
\centering 
 
\begin{tikzpicture}[x=1pt, y=1pt] 
\filldraw[opacity=0.3] (110,10)--(110,50)--(150,90)--(190,90)--(190,50)--(150,10)--(110,10);
\draw[very thick] (150,50)--(200,100); 
\draw[very thick] (150,50)--(100,50);
\draw[very thick] (150,50)--(150,0);

\begin{scope}[xshift=140]
\filldraw[opacity=0.3] (110,10)--(110,50)--(110,90)--(230,90)--(230,10)--(110,10);

\draw[very thick] (150,50)--(200,100); 
\draw[very thick] (150,50)--(100,50);  
\draw[very thick] (150,50)--(150,0);  
\end{scope} 

\end{tikzpicture} 

\caption{Regions $V$ with kinks (left) and without kinks (right) along the tropical line.} 
\label{fig:kink} 
\end{figure} 

\begin{rmk}
\label{rmk:kinks}  
There is an important issue which we have glossed over in this computation: in order for \eqref{eqn:typeII} to hold, the boundary of the region $V$ should be smooth and transverse to the edges of the discriminant locus (i.e., the legs of the tropical line) where it crosses them. 
For example, if one takes the region $V$ shown in the left side of Figure \ref{fig:kink}, the value of the integral \eqref{eqn:typeII} will be equal to $5\zeta(3)/4$ (see the proof of Proposition \ref{prop:asymptotics_of_integral}). 
Some of the contribution of the vertex is `hiding in the kinks in the boundary of $V$' in this case. 
It turns out that in higher dimensions, the local contribution is even more strongly dependent on the shape of the region $V$. 
For example it is not enough that $V$ have no `kinks' where it crosses the discriminant locus: in dimensions $\ge 4$ the integral may in general depend on the angle at which $V$ intersects the singular locus. We have not found a way to organize these choices efficiently.
In Section \ref{sec:periodR} we take the more pedestrian approach of decomposing the cycle into pieces in a completely canonical way, at the cost of leaving certain `kinks' in the pieces which result in a formula which is less visibly `local' in the base of the SYZ fibration. 
\end{rmk}

\subsubsection{Proof of Theorem \ref{thm:gamma} in dimension 3}

We can piece together a sketch proof of Theorem \ref{thm:gamma} in dimension $3$ from the pieces we have assembled. 
In Section \ref{sec:leading} we have seen that the codimension-$0$ strata of the base of the SYZ fibration contribute the $i=0$ term on the right-hand side of \eqref{eqn:gammai}; in Section \ref{sec:edges} we have seen that the edges (codimension-$2$ strata) contribute the $i=2$ term; it remains to see how the type I and type II vertices contribute the $i=3$ term. 
It is clear that their contribution is
\[ \left( \#(\text{type I vertices}) - \#(\text{type II vertices})\right) \cdot \zeta(3),\]
which we must show is equal to
\[ \int_X \hGamma_3 = -2\zeta(3) \cdot \int_X \ch_3(TX) = -\zeta(3) \cdot \int_X c_3(TX) = -\zeta(3) \cdot \chi(X). \]
The answer now follows from the observation that in the stratification of $X$ according to singularities of the SYZ fibration, each stratum has an $S^1$ factor and therefore vanishing Euler characteristic except for those lying over the vertices of the discriminant locus. 
The mirror to a type I SYZ fibre is a type II SYZ fibre, which has Euler characteristic $-1$; whereas the mirror to a type II SYZ fibre is a type I SYZ fibre, which has Euler characteristic $+1$. 
Therefore we have
\[ \chi(X) =  \#(\text{type II vertices}) - \#(\text{type I vertices}),\]
 which completes the sketch of a proof.

\begin{example}
The SYZ fibration on the quintic threefold has $\binom{5}{3} \cdot 5 = 50$ vertices of type I and $\binom{5}{2} \cdot 5^2 = 250$ vertices of type II. 
The Euler characteristic of the mirror quintic is $250-50 = 200$, as one clearly sees from its Hodge diamond \cite{CdlOGP:pair}.
\end{example}

\section{Proof of Theorem \ref{thm:periodR}}
\label{sec:periodR}
In this section we prove Theorem \ref{thm:periodR}. We break up the mirror period integral into pieces corresponding to a 
polyhedral decomposition of $\partial \Delta_\lambda$, which is the limit shape of $C_t^+$. Then we express each piece in terms of integrals over the toric variety $Y_{\Delta_\lambda}$ by applying the Duistermaat--Heckman theorem.   

\subsection{Tropical setup}
\label{subsec:tropical_setup} 

Consider the affine functions
\begin{align*}
\beta_q\colon P_\R  \to \R,  \qquad 
\beta_q(p) = \langle q,p \rangle + \lambda_q,
\end{align*}
as well as the map
\begin{align*}
i_t\colon P_\R \to P_{\C^*}, \qquad 
i_t(p_0,\dots,p_n) = \left(t^{p_0},\dots,t^{p_n}\right)
\end{align*}
for $t \in \R_+$, which is right-inverse to $\Log_t$.
If we define 
\[B_t := \left\{ 1 = \sum_{q \in V} t^{\beta_q}\right\} \subset P_\R,\]
then it is clear that $i_t(B_t) = C_t^+$. 
Therefore
\[ \int_{C_t^+} \Omega_t = \int_{B_t} i_t^* \Omega_t.\]

We observe that
\[ 
\log_t\left|t^{\lambda_q} \cdot z^q\right| = \beta_q(\Log_t(z)) 
\]
so $\beta_q$ is the `tropical monomial' corresponding to the honest monomial $t^{\lambda_q} \cdot z^q$. 
As a result, in the limit $t \to 0$, the \emph{amoeba} $\Log_t(\Zring_t)$  converges to the \emph{tropical amoeba} $\Sing(\min(0, \{\beta_q\}_{q\in V}))$, the non-smooth locus of the piecewise affine-linear function $\min(0,\{\beta_q\}_{q\in V})$ on $P_\R$  \cite{Mikhalkin2004}.
We observe that the unique compact component of the complement of the tropical amoeba is precisely the polytope $\Delta_\lambda$ that appears on the $A$-side of our mirror statement. In the limit $t\to 0$, $B_t$ converges to the boundary $\partial\Delta_\lambda$ of the polytope.  

\subsection{Decomposing the domain}
\label{sec:decomposing-domain}

We now decompose the domain $B_t$ of our period integral into regions where the different monomials dominate. 

\begin{figure}[ht]
\centering
\begin{tikzpicture}[x=1pt, y=1pt]

\filldraw[gray!30!white] (148,30)--(144,36)--(187,36)--(181,30)--(148,30); 
\filldraw[gray] (148,30)--(144,36) -- (136,36) -- (140,30) -- (148,30); 

\draw[very thick] (100,90)--(200,90)--(220,60)--(190,30)--(140,30)--(100,90);
\draw[very thick] (233,25)--(196,-12)--(124,-12)--(80,55);
\draw[very thick] (100,90)--(80,55); 
\draw[very thick] (220,60)--(233,25);
\draw[very thick] (190,30)--(196,-12);
\draw[very thick] (140,30)--(124,-12); 

\draw[dotted] (200,90)--(213,55); 
\draw[dotted] (80,55)--(213,55)--(233,25); 

\draw[very thin] (148,30)--(108,90);
\draw[very thin] (136,36)--(196,36);
\draw[very thin] (104,84)--(204,84);
\draw[very thin] (193,90)--(215,55);
\draw[very thin] (181,30)--(216,65);

\draw (160,62) node {face $q$}; 
\draw (165,10) node {face $k$};  
\draw (115,39) node {\rotatebox{-56}{face $j$}};

\end{tikzpicture}

\caption{Decomposition of the cycle $B_t$. The cycle $B_t$ approaches to the boundary of the polytope $\Delta_\lambda$ as $t\to +0$. The light grey region is the limit of $B_t^{q,\{k\}}$ and the dark grey region is the limit of $B_t^{q,\{k,j\}}$; these pieces $B_t^{q,K}$ can be presented as graphs over the shaded regions.}
\label{fig:piece}
\end{figure}
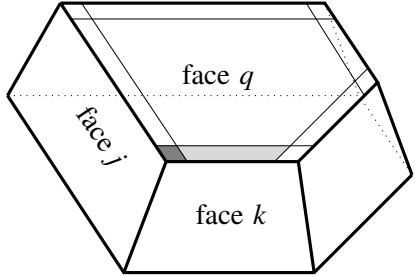 

We cover $P_\R$ with the sets 
\[ U^q := \left\{p \in P_\R: \beta_i(p) \ge \beta_q(p) \text{ for all $i \in V$}\right\}\]
for $q \in V$. Thus we can cover $B_t$ with the sets $B^q_t := U^q \cap B_t$. 
In the limit $t \to 0$, $B_t^q$ converges to the $q$th face of $\Delta_\lambda$.

The above cover is well-adapted to consider tropical limits, but our analysis of sub-leading terms requires a further decomposition.  
Let us fix $\epsilon >0$, and for each $\{q\} \sqcup K \subset V$ set
\[ 
U^{q,K} := \left\{ p \in P_\R : 
\begin{aligned}
\beta_k(p) - \beta_q(p) & \in [0,\epsilon] & &\text{for $k\in \{q\}\cup K$} \\ 
\beta_m(p )- \beta_q(p) &\in [\epsilon, \infty) & & \text{for $m\in V\setminus (\{q\}\cup K)$}
\end{aligned}
\right\} 
\]
In words, $U^{q,K}$ is the region where the tropical monomial $\beta_q$ is smallest (hence `dominates') and the tropical monomials $\{\beta_k\}_{k \in K}$ are not far behind.

We observe that $U^q$ is covered by the sets $U^{q,K}$. 
Then we obtain a cover $B_t^{q,K} := U^{q,K} \cap B_t$ (see Figure \ref{fig:piece}). 
So our period integral is equal to
\[ \sum_{q,K} \int_{B_t^{q,K}} i_t^* \Omega_t.\]
If we choose $\epsilon>0$ small enough, then $B_t^{q,K}$ is nonempty for sufficiently small $t>0$ if and only if the facets $\{\beta_i=0\}\cap \Delta_\lambda$ with $i \in\{q\}\sqcup K$ have nonempty intersection,  or equivalently, $\{q\}\sqcup K$ spans a cone of the fan $\Sigma_\lambda$. Starting in the next section, and going through the end of Section \ref{sec:duistermaat-heckman}, we will restrict to pairs $(q,K)$ such that the facets corresponding to $\{q\}\sqcup K$ intersect.

\subsection{Approximation in each region}
\label{sec:approx_each_region} 
Let us consider the integral over $B_t^{q,K}$. 
Observe that 
\begin{align}
\label{eq:f_t_on_region}
f_t(z) &= t^{\lambda_q}z^q \left (1 + \sum_{k\in K} t^{\lambda_k-\lambda_q} z^{k-q} + h_t(z)\right),
\end{align}
where
\[ 
h_t(z) = \sum_{j\in V \setminus (\{q\}\sqcup K)} t^{\lambda_j-\lambda_q} \cdot z^{j-q}.
\]
We observe that over $i_t(B_t^{q,K})$, we have
\begin{equation}
\label{eqn:negl} h_t, \ z_i\frac{\partial h_t}{\partial z_i} \in \bigO
\end{equation}
because each contributing monomial is so. 
The idea for approximating the integral over $B_t^{q,K}$ is to `throw away' these negligible terms.

In order to evaluate the integral over the region $B_t^{q,K}$, we introduce an affine coordinate system $(a,b_k,c_j)$ on $P_\R$: 
\begin{align*}
a &= \beta_q \\
b_k &= \beta_k - \beta_q \quad\text{for $k \in K$} \\
\{c_j\} &= \text{a collection of integral linear functions completing a coordinate system}.
\end{align*}
The fact that it is possible to complete $\{a,b_k\}$ to a coordinate system follows from our assumption that the fan $\Sigma_\lambda$ is simplicial and that the facets corresponding to $\{q\}\sqcup K$ intersect. We also write 
\[
r_{q,K} \cdot da \prod_{k\in K} db_k \prod_j dc_j = \text{the standard affine volume form on $P_\R$}
\]
for some factor $r_{q,K}>0$ and set 
\begin{equation} 
\label{eq:residual_vol} 
d\vol_{q,K}= r_{q,K} \prod_{j} dc_j 
\end{equation} 
for the residual volume form on the $c$-plane.  When $a$, $b_k$ already form a coordinte system and there are no $c$ variables, we regard $d\vol_{q,K}$ as a measure on the point $\{0\}=\R^0$. Note that this is different from the affine volume form induced on a subspace of the form $\{a=\text{const}, b_k= \text{const}\}$ unless the covectors $da=q$, $db_k=k-q$ are part of a $\Z$-basis of $P^\vee$. 

We introduce the corresponding monomials on $P_{\C^*}$:
\begin{align*}
w &= t^{\lambda_q} \cdot z^q \\
x_k &= t^{\lambda_k - \lambda_q} \cdot z^{k-q} \quad \text{ for $k \in K$} \\
y_j &= z^{c_j}.
\end{align*}
In these coordinates we have
\[ 
f_t(w,x,y) = w \cdot \left( 1+\sum_{k \in K} x_k + h_t(x,y,w)\right).
\]
By the definition of $r_{q,K}$, the standard holomorphic volume form 
on $P_{\C^*}$ is given by 
$r_{q,K}\cdot d\log w\wedge \bigwedge_{k\in K}
d\log x_k \wedge \bigwedge_j d\log y_j$. 
Thus the volume form of $\Zring_t$ is
\begin{align}
\label{eq:Omega_in_xy} 
\begin{split}
\Omega_t &= r_{q,K} \cdot \left. \frac{d\log w \wedge \bigwedge_{k\in K}d\log x_k \wedge \bigwedge_j d\log y_j}{df_t(w,x,y)} \right|_{\Zring_t} \\
&= r_{q,K}\cdot \left. \frac{\bigwedge_{k\in K}d\log x_k \wedge \bigwedge_j d\log y_j}{w\cdot  (\partial f_t(w,x,y)/\partial w)} \right|_{\Zring_t}  
\end{split} 
\end{align}
in the region where the denominator does not vanish. 

On $C_t^+=i_t(B_t)=\{z\in P_{\R_+} : 1=f_t(z)\}$, we have
\begin{align*}
\label{eq:estimate_wf_w}
w \cdot \frac{\partial f_t}{\partial w} (w,x,y)
&= f_t(w,x,y) + w^2 \frac{\partial h_t}{\partial w}(w,x,y) && \\ 
& = 1 + \bigO && \text{over $i_t(B_t^{q,K})$}
\end{align*}
where we used \eqref{eqn:negl} and the fact that $0<w<f_t(x,y,w)=1$ on $C_t^+$.  
Therefore we have 
\begin{equation*} 
\Omega_t = \left( 1 + \bigO \right) r_{q,K} \cdot \bigwedge_{k\in K}d\log x_k \wedge \bigwedge_j d\log y_j, 
\end{equation*} 
so 
\begin{align}
\label{eq:Omega_in_bc}
\begin{split} 
\int_{B_t^{q,K}} i_t^* \Omega_t & = \left( 1 + \bigO \right) \cdot r_{q,K}\cdot \int_{B_t^{q,K}} \bigwedge_{k\in K} d\log\left(t^{b_k}\right) \wedge \bigwedge_j d\log\left(t^{c_j}\right)\\
&= \left( 1 + \bigO \right) \cdot (-\log t)^n \cdot  \vol\left(\pi_{b,c}\left(B_t^{q,K}\right)\right),
\end{split} 
\end{align}
where $\pi_{b,c}$ denotes the projection to the $(b,c)$-plane and $\vol$ denotes the volume with respect to the product of $\prod_{k\in K} db_k$ and the residual volume form $d\vol_{q,K}$ in \eqref{eq:residual_vol}. 
See the remark below for the reason why $(-\log t)^n$ instead of $(\log t)^n$  
	appears in the last expression.

\begin{rmk} 
\label{rmk:sign_orientation} 
We have been vague about how we choose an order of the coordinates $a$, $b_k$, $c_j$ (or $w$, $x_k$, $y_j$) and an orientation of the cycle $B_t$; strictly speaking we need them to define $\Omega_t$ and the integral. For convenience, we shall always arrange these choices so that $i_t^*\Omega_t$ defines a positive measure (density) on $B_t$. Note that the factor $(-\log t)^n$ appearing in the above formula is positive since $(-\log t)>0$ for a sufficiently small $t$. 
\end{rmk} 

\subsection{Approximation in terms of volumes of polytopes} 
 \label{sec:approx_vol_polytope}
We now approximate the affine volume of $\pi_{b,c}(B_t^{q,K})$ in terms of the volumes of certain polytopes. 

On $B_t^{q,K}$, the defining equation can be rewritten as 
\begin{align}
\label{eq:defining_eqn_Bt}
1 = \sum_p t^{\beta_p} \iff 1 = t^a \cdot \bigg(1+\sum_{k \in K} t^{b_k} + \underbrace{\sum_{m \in V \setminus (\{q\} \sqcup K)} t^{\beta_m-a}}_{O(t^\epsilon)}\bigg) ,
\end{align}
which can be used to write $a$ as a function $a_{q,K}(b,c)$ of the variables $b_k,c_j$. 
We observe that we have the approximation
\begin{equation}
\label{eq:approx_aqK}
a_{q,K}(b,c) = a'_{q,K}(b) + \bigO \quad \text{ where} \quad 
a'_{q,K}(b) := -\log_t\left(1+\sum_{k \in K} t^{b_k} \right).
\end{equation}

Now $B_t^{q,K}$ is defined by the inequalities
\begin{align*}
b_k & \in [0,\epsilon] && \text{ for $k \in K$} \\
\beta_m(a,b,c) - a & \in [\epsilon,\infty)&& \text{ for $m \notin \{q\} \sqcup K$}, \\
\intertext{which means the region $\pi_{b,c}\left(B_t^{q,K}\right)$ is defined by the inequalities}
b_k & \in [0,\epsilon] && \text{ for $k \in K$} \\
\beta_m(a_{q,K}(b,c),b,c) - a_{q,K}(b,c) & \in [\epsilon,\infty) && \text{ for $m \notin \{q\}\sqcup K$}.
\end{align*}
 
We will consider the fibres $F_{q,K}(b)$ of the projection
\[ \pi_b: \pi_{b,c}\left(B_t^{q,K}\right) \to [0,\epsilon]^K.\]
It is clear that 
\[ \vol\left(\pi_{b,c}\left(B_t^{q,K}\right)\right) = \int_{[0,\epsilon]^K} \vol\left(F_{q,K}(b)\right) db,\]
where we use the volume form $d\vol_{q,K}$ \eqref{eq:residual_vol} on the $c$-plane to define $\vol(F_{q,K}(b))$, 
so our next project is to approximate the volume of the fibres $F_{q,K}(b)$. 
We claim that 
\[ 
\vol(F_{q,K}(b)) = \vol(F'_{q,K}(a'_{q,K}(b),b)) + \bigO
\]
where $F'_{q,K}(a,b)$ is the compact polytope in the $c$-plane defined by 
\[
\beta_m(a,b,c) - a  \in [\epsilon,\infty) \quad \text{ for $m \notin \{q\}\sqcup K$}
\]
with fixed $(a,b)$. 
Indeed, this follows because $F_{q,K}(b)$ can be sandwiched between two perturbations of the compact polytope $F'_{q,K}(a'_{q,K}(b),b)$ where the facets have been shifted by quantities of order $\bigO$.

We have succeeded in approximating the volume of $\pi_{b,c}(B_t^{q,K})$ in terms of 
the volumes of the polytopes $F'_{q,K}(a,b)$, but we would prefer to work with the volumes of the polytopes $E_{q,K}(a,b)$ defined by
\[
\beta_m(a,b,c) - a  \ge 0 \quad \text{ for $m \notin \{q\} \sqcup K$}
\]
with fixed $(a,b)$. 
We shall regard $F'_{q,K}(a,b)$ and $E_{q,K}(a,b)$ either as polytopes in the $c$-plane or as subsets of $P_\R$ with the values of $(a,b)$ fixed. 
More generally, for any subset $K' \subset V$ not containing $q$, we write 
$E_{q,K'}(a,\{b_k\}_{k\in K'})$ for the polytope in $P_\R$ defined by 
\[
\beta_m - a \ge 0 \quad \text{ for $m\notin \{q\} \sqcup K'$}
\]
with the values of $a=\beta_q$ and $b_k =\beta_k - \beta_q$ (with $k\in K'$) fixed.
We have 
\[
F'_{q,K}(a,b)=E_{q,K}(a,b) \setminus \bigcup_{j\notin \{q\}\sqcup K} \bigcup_{b_j \in [0,\epsilon]} E_{q,K\sqcup\{j\}}(a,b,b_j).
\]
This volume can be computed by the inclusion-exclusion principle: noting that 
\[
\bigcap_{j\in I} \bigg(\bigcup_{b_j \in [0,\epsilon]} E_{q, K\sqcup\{j\}}(a,b,b_j)\bigg)
= \bigcup_{b'\in [0,\epsilon]^I} E_{q,K\sqcup I}(a,b,b') 
\] 
for $I$ disjoint from $\{q\}\sqcup K$, we obtain 
\[ 
\vol(F'_{q,K}(a,b)) =  \sum_{J : J \supset K, q\notin J} (-1)^{|J \setminus K|} \cdot \int_{[0,\epsilon]^{J \setminus K} } \vol(E_{q,J}(a,b,b')) db' 
\]
where we write $b=(b_k)_{k\in K}$, $b'=(b_j)_{j\in J\setminus K}$, and use the volume form $d \vol_{q,J}$ to define $\vol(E_{q,J}(a,b,b'))$. 
This means our period integral becomes 
\begin{align}
\label{eqn:perH} 
\begin{split}
\int_{C_t^+} \Omega_t  & = \left(1+\bigO\right) (-\log t)^n \left(\sum_{q,K,q\notin K}  \int_{[0,\epsilon]^K} \vol\left(F'_{q,K}\left(a'_{q,K}(b),b\right)\right) db + \bigO\right) \\
&= \left(1+\bigO\right)  (-\log t)^n  \\
& \qquad \times \left(\sum_{q,K \subset J, q\notin J} (-1)^{|J \setminus K|} \int_{[0,\epsilon]^J} \vol\left(E_{q,J}\left(a'_{q,K}(b),b,b'\right)\right) db db' + \bigO \right).  
\end{split}
\end{align}

\subsection{Duistermaat--Heckman} 
\label{sec:duistermaat-heckman}

We apply the Duistermaat--Heckman theorem to express the volumes of polytopes in \eqref{eqn:perH} as symplectic volumes.

\begin{lem} 
\label{lem:DH}
For positive, sufficiently small $a$ and $b_j$ with $j\in J$, we have 
\[ 
\vol(E_{q,J}(a,b)) = \int_{Y_{\Delta_\lambda}} \exp\left(\omega_\lambda - \sum_{j \in J} b_j \cdot D_j -a  \cdot \sigma \right) \cdot D_q \cdot \prod_{j \in J} D_j ,
\]
where $D_j \subset Y_{\Delta_\lambda}$ denotes the toric divisor corresponding to the $j$th facet $\{\beta_j=0\}\cap \Delta_\lambda$ of $\Delta_\lambda$, and $\sigma:=\sum_{j \in V} D_j$. 
The right-hand side vanishes when the facets corresponding to the elements  of $\{q\} \sqcup J$ do not intersect. 
\end{lem}
\begin{proof}
We use the Duistermaat--Heckman theorem to identify the volume of $E_{q,J}(q,b)$ with the symplectic volume of a toric subvariety of $Y_{\Delta_\lambda}$. 
The polytope $E_{q,J}(a,b)$ is defined by
\begin{align*}
\beta_q &= a \\
\beta_j - a &= b_j && \text{for $j \in J$} \\
\beta_m - a & \ge 0 && \text{for $m \in V\setminus (\{q\}\sqcup J)$},\\
\intertext{which is equivalent to}
\beta_q - a & =0 \\
\beta_j - b_j - a &= 0 && \text{for $j \in J$} \\
\beta_m - a  & \ge 0 && \text{for $m \in V\setminus (\{q\} \sqcup J)$}. 
\intertext{
This is precisely the face of the polytope $\Delta_{\lambda'}$ corresponding to the set $\{q\} \sqcup J$, where} 
\lambda'_j &= \lambda_j - b_j - a && \text{for $j \in J$} \\
\lambda'_m &= \lambda_m - a && \text{for $m \in V\setminus J$}.
\end{align*}
When $a$ and $\{b_j\}_{j\in J}$ are sufficiently small, the combinatorial type of $\Delta_{\lambda'}$ is the same as that of $\Delta_\lambda$, and the volume of the face corresponding to $\{q\} \sqcup J$ is equal to the symplectic volume of the stratum 
\[ D_q \cap \bigcap_{j \in J} D_j\]
with respect to a symplectic form in cohomology class
\[ \left[\omega_{\lambda'}\right] = \left[\omega_\lambda \right] - \sum_{j \in J} b_j \cdot D_j - a \cdot \sum_{j \in V} D_j,\]
by \cite[Theorem 2.10]{Guillemin1994}.
This yields the result. The right-hand side vanishes if $D_q \cap \bigcap_{j\in J} D_j= \emptyset$, and therefore if the facets of $\Delta_\lambda$ from $\{q\}\sqcup J$ do not intersect. 
\end{proof}
\begin{rmk} We hid some technical details when applying the Duistermaat--Heckman theorem. When the covectors $da$, $db_j$ are not part of a $\Z$-basis of $P^\vee$, the corresponding toric substack $D_q\cap \bigcap_{j\in J}D_j$ has a generic stabilizer. The order of the generic stabilizer equals the ratio between the affine volume form of the face corresponding to $\{q\}\sqcup J$ and the residual volume form $d\vol_{q,J}$ on the $c$-plane. Since, by definition, the integral over $D_q \cap \bigcap_{j\in J} D_j$ is the integral over the coarse moduli space divided by the order of the generic stabilizer, the volume of $E_{q,J}(a,b)$ with respect to $d\vol_{q,J}$ gives the correct answer.
\end{rmk} 

We now substitute this into \eqref{eqn:perH}: we can ensure that $b_j$ in \eqref{eqn:perH} is sufficiently small by making $\epsilon>0$ small, and also ensure that $a'_{q,K}(b)$ in \eqref{eqn:perH} is sufficiently small by making  $t>0$ small because of the estimate: 
\[
0\le a'_{q,K}(b)= \frac{\log(1+\sum_{k\in K} t^{b_k})}{-\log t} \le \frac{\log |V|}{-\log t}. 
\] 
We now obtain: 
\begin{align*}
\int_{C_t^+} \Omega_t &= \left(1 + \bigO\right) \cdot \left(  \int_{Y_{\Delta_\lambda}} P_t(D)_{n+1} + O\left((-\log t)^n  t^\epsilon\right) \right),
\end{align*}
where 
\begin{align*}
P_t(D) &= (-\log t)^n \cdot \sum_{\substack{q,K \subset J \\ q\notin J}} (-1)^{|J \setminus K|} \int_{[0,\epsilon]^J}  e^{\omega_\lambda - \sum_{j \in J} b_j \cdot D_j + \sigma \cdot \log_t\left(1+\sum_{k \in K} t^{b_k}\right)} db \cdot D_q \prod_{j \in J} D_j
\end{align*}
The subscript `$n+1$' denotes the part of $P_t(D)$ in degree $2(n+1)$: that is the only part that gets hit by the integral $\int_{Y_{\Delta_\lambda}}$. The summand for $(J,q)$ automatically vanishes unless the facets corresponding to $\{q\} \sqcup J$ intersect, in particular, unless the facets corresponding to $\{q\} \sqcup K$ intersect. Therefore we can now withdraw the assumption imposed at the end of Section \ref{sec:decomposing-domain} that the facets corresponding to $\{q\}\sqcup K$ intersect and consider the sum over arbitrary $K, J, q$ with $K\subset J$ and $q\notin J$.

\subsection{End of the proof} 

By replacing $\omega_\lambda$, $D_j$, $\sigma$ 
in the definition of $P_t(D)$ with $(-\log t) \omega_\lambda$, $(-\log t) D_j$, $(-\log t) \sigma$ 
respectively, we obtain 
\begin{align*} 
(-\log t)^n \sum_{q,K\subset J, q\notin J} 
(-1)^{|J\setminus K|}&  \int_{[0,\epsilon]^J} 
e^{(-\log t)(\omega_\lambda - \sum_{j\in J} b_j \cdot D_j)  
- \sigma \log (1+\sum_{k\in K} t^{b_k})}db \\ 
& \times (-\log t)^{|J|+1} D_q \prod_{j\in J} D_j 
\end{align*} 
The degree $2(n+1)$ part of this quantity equals $(-\log t)^{n+1} P_t(D)_{n+1}$. 
Making the substitution $s_j = -\log t \cdot b_j$, 
we find that $P_t(D)_{n+1} = Q_t(D)_{n+1}$ with 
\[
Q_t(D) = \sum_{q,K \subset J,q\notin J} (-1)^{|J \setminus K|} \cdot \int_{[0,-\epsilon \log t]^J} t^{-\omega_\lambda} \cdot e^{-\sum_{j \in J} s_j \cdot D_j - \sigma\cdot \log\left(1+\sum_{k \in K} e^{-s_k} \right)} ds \cdot D_q\prod_{j \in J} D_j 
\]
where the factor $(-\log t)^{|J|}$ is absorbed by $d b$ to become $d s$. 
By expanding the exponential, we find that
\begin{align*}
Q_t(D) 
&= t^{-\omega_\lambda} \left(\sigma+ \sum_{q,J,\ell,\vec{m}: q\notin J, J\neq \emptyset} \frac{I^\epsilon_{\ell;\vec{m}}(t)}{\ell! \prod_{j\in J}m_j!}  \cdot D_q \cdot  (-\sigma)^\ell  \prod_{j \in J} (-D_j)^{m_j + 1}\right),
\end{align*}
where the sum is over $\ell \in \Z_{\ge 0}$, $\vec{m} \in (\Z_{\ge 0})^J$, $q$, $J$ with $q\notin J$ and 
\[ 
I_{\ell;\vec{m}}^\epsilon(t) :=  \int_{[0,-\epsilon \log t]^J} \prod_{j\in J} s_j^{m_j} \cdot g_\ell\left(\{e^{-s_i}\}_{i \in J}\right) ds  
\]
is an $\epsilon$-truncated version of the `local integral' $I_{\ell,\vec{m}}$ in \eqref{eqn:Ilm}. 
We note that the first term $\sigma=\sum_{q\in V} D_q$ arises from the case where $J=K=\emptyset$. 
\begin{lem} 
	\label{lem:convergence_Ilm}We have $I^\epsilon_{\ell;\vec{m}}(t) = I_{\ell;\vec{m}} + O\left((-\log t)^{|\vec{m}|} t^\epsilon\right)$ as $t\to +0$, where $|\vec{m}|=\sum_{j\in J} m_j$. 
\end{lem}
\begin{proof} 
We recall the bound \eqref{eqn:boundint}, which was used to prove exponential decay of the integrand at infinity. It gives
\[\prod_{j \in J}s_j^{m_j} \cdot g_\ell\left(\{e^{-s_i}\}_{i \in J}\right) \le C_\ell \cdot \prod_{j \in J} e^{-s_j} s_j^{m_j} \qquad\text{for $s_j \ge 0$.}\] 
The order estimate then follows by 
\[
\int_{-\epsilon \log t}^\infty s^{m} e^{-s} ds = O\left((-\log t)^m t^\epsilon\right).
\]
\end{proof} 

Now observe that $I_{0,\vec{m}} = 0$, and the anticanonical hypersurface $X$ is homologous to the toric boundary divisor $\sigma=\sum_{q\in V}D_q$, so we have proved
\[
 \int_{C_t^+} \Omega_t = \left(1 + \bigO\right) \cdot \left(\int_X t^{-\omega_\lambda} \cdot \hG_X + O\left((-\log t)^n t^\epsilon\right) \right) 
\]
where $\hG_X$ is as given in \eqref{eq:hG_X}. 
Because $-\log t = O(t^{-\delta})$ for any $\delta >0$, we can absorb the error terms depending on $\log t$ by reducing $\epsilon$, and thereby obtain
\[ 
\int_{C_t^+} \Omega_t = \int_X t^{-\omega_\lambda} \cdot \hG_X + \bigO
\]
for some (new, smaller) $\epsilon > 0$. This completes the proof of Theorem \ref{thm:periodR}.

\section{Proof of Theorem \ref{thm:gamma}}
\label{sec:gamma}

\subsection{Formula for the Gamma class}

Since the Gamma class is multiplicative, the short exact sequence $0 \to TX \to TY \to NX \to 0$ gives
\begin{align*}
\hGamma(TX) &= \frac{\hGamma(TY)}{\hGamma(NX)}.
\end{align*}
The Euler sequence on the toric variety $Y$ gives the following expression for its Gamma class:
\[ \hGamma(TY) = \prod_{j \in V} \Gamma(1+D_j)\]
(compare Cox, Little and Schenck \cite[Proposition 13.1.2]{CLS}). 
Setting $\sigma := \sum_{j\in V} D_j$ as before, we have
\[ \hGamma(NX) = \Gamma(1+\sigma)\]
because $X$ is anticanonical and $K_Y = -\sigma$. 
Thus we have the formula
\begin{equation}
\label{eq:1Gamma_X}
\hGamma_X = \frac{\prod_{j \in V} \Gamma(1+D_j)}{\Gamma(1+\sigma)}.
\end{equation} 
Substituting in the power series expansion of $\Gamma(1+z)$, we obtain the more explicit
\begin{equation}
\label{eq:2Gamma_X}
\hGamma_X = \exp\left(\sum_{k \ge 2} (-1)^k \cdot \frac{\zeta(k)}{k} \cdot \left(\sum_{j \in V} D_j^k - \sigma^k\right)\right).
\end{equation}

\subsection{The identity $\hG_X = \hGamma_X$ as formal power series} 
\label{subsec:G_Gamma} 

The expressions \eqref{eq:hG_X}, \eqref{eq:2Gamma_X} for $\hG_X$ and $\hGamma_X$ respectively define symmetric formal power series in the variables $D_j$, $j\in V$. Theorem \ref{thm:gamma} follows from the following stronger statement: 
\begin{prop} \label{pro:G_Gamma}We have $\hG_X = \hGamma_X$ as formal power series in $\{D_j:j\in V\}$. 
\end{prop} 
In the rest of this Section \ref{subsec:G_Gamma}, we prove Proposition \ref{pro:G_Gamma}. By \eqref{eq:hG_X}, we have 
\begin{multline*} 
\hG_X= 1 + \sum_{q,J,\ell,\vec{m}} (-D_q) \prod_{j\in J}(-D_j) \\
\times \int_{[0,\infty)^J} \frac{(-\sigma)^{\ell-1} \prod_{j\in J}(-s_j D_j)^{m_j}}{\ell! \prod_{j\in J} m_j!}
\sum_{K\subset J} (-1)^{|K|} \log^\ell\left(1+\sum_{k\in K} e^{-s_k}\right) ds^J
\end{multline*} 
where, as before, the sum is taken over all $\ell \ge 1$, $q\in V$, all nonempty subsets $J\subset V$ with $q\notin J$, and all vectors $\vec{m} \colon J \to \Z_{\ge 0}$, and we write $ds^J = \prod_{j\in J} ds_j$. We now regard $D_j$ as positive real numbers and introduce the following function $G_X(D)$ of $D=(D_j:j\in V) \in (\R_+)^V$: 
\begin{multline*} 
G_X(D) := 1+\sum_{q,J} (-D_q) \prod_{j\in J}(-D_j) \\
\times \int_{[0,\infty)^J} e^{-\sum_{j\in J} D_j s_j} 
\sum_{K\subset J} (-1)^{|K|} \frac{\left(1+\sum_{k\in K} e^{-s_k}\right)^{-\sigma}-1}{-\sigma} ds^J 
\end{multline*}
where the sum is over all $q\in V$ and all nonempty subsets $J \subset V$ not containing $q$, and $\sigma = \sum_{j\in V} D_j$. The convergence of the integral is ensured by the exponentially decaying factor $e^{-\sum_{j\in V} D_j s_j}$. 

It is straightforward to compute that, if the Taylor expansion of the integrand could be exchanged with the integral in the definition of $G_X(D)$, the result would be the formal power series $\hG_X$. 
In fact we prove in Lemma \ref{lem:asymp} below that, for a fixed $D=(D_j:j\in V) \in (\R_+)^V$, we have the asymptotic expansion
\begin{equation}
  \label{eq:asymptotic_expansion_G_X}
G_X(y D) \sim \left. \hG_X\right |_{D_j \rightarrow y D_j} \qquad \text{as $y  \to +0$} 
\end{equation}
where $\hG_X|_{D_j \to y D_j}$ means the substitution of $yD_j$ for $D_j$ 
in the formal power series $\hG_X$.  

Similarly, we have the asymptotic expansion\footnote{actually the Taylor expansion}
\[
\Gamma_X(y D) \sim \left. \hGamma_X\right |_{D_j\to y D_j} \qquad \text{as $y\to +0$}
\]
where $\Gamma_X(D)$ is given by 
\[
\Gamma_X(D) := \frac{\prod_{j\in V}\Gamma(1+D_j)}{\Gamma(1+\sigma)} \qquad \text{with $\sigma=\sum_{j\in V} D_j$}. 
\] 
Therefore it suffices to show that $G_X(D) = \Gamma_X(D)$ as functions of $D$.
 
Note that we can interchange the integral sign with the sum over $K$ in the definition of $G_X(D)$ because of the factor $e^{-\sum_{j\in J} D_j s_j}$ (this interchange was not possible for $\hG_X$). Thus $G_X(D)$ equals: 
\begin{align*}
& 1+\sum_{q,J} (-D_q) \prod_{j\in J}(-D_j) 
\int_{[0,\infty)^J} e^{-\sum_{j\in J} D_j s_j} 
\sum_{K\subset J} (-1)^{|K|} \frac{\left(1+\sum_{k\in K} e^{-s_k}\right)^{-\sigma}}{-\sigma} ds^J \\
&= 1+ \sum_{K\subset J, J\neq \emptyset, q\notin J}
(-1)^{|J\setminus K|}\frac{D_q\prod_{k\in K} D_k}{\sigma} 
\int_{[0,\infty)^K} e^{-\sum_{k\in K} D_k s_k}\left(1+\sum_{k\in K} e^{-s_k}\right)^{-\sigma} ds^K. 
\end{align*} 
In the first line we used the fact that $\sum_{K\subset J}(-1)^K = 0$, and in the second line we interchanged the integration and  summation, and then integrated $s_j$ out for $j\in J\setminus K$. 
Fixing an element $q\in V$ and a subset $K\subset V$ not containing $q$, we sum over subsets $J$ containing $K$ but not $q$. Using the fact that 
\[
\sum_{K\subset J\subset V\setminus \{q\}, J\neq \emptyset}
(-1)^{|J\setminus K|} = 
\begin{cases} 
-1 & \text{if $K=\emptyset$;}\\
0 & \text{if $|K|\le |V|-2$;} \\
1 &\text{if $|K|=|V|-1$,} 
\end{cases}
\]
we obtain 
\begin{equation} 
	\label{eq:GXD}
G_X(D) =\frac{\prod_{j\in V} D_j}{\sigma} 
\sum_{K\sqcup \{q\} = V} 
\int_{[0,\infty)^K}e^{-\sum_{k\in K} D_ks_k}\left(1+\sum_{k\in K}e^{-s_k}\right)^{-\sigma} ds^K, 
\end{equation} 
where the case $K=\emptyset$ cancels the leading term $1$ and only the case $|K|=|V|-1$ remains.

In order to compute the sum of integrals in Equation \eqref{eq:GXD}, we interpret the domains of integration as subsets of the projective space over the tropical numbers: concretely, we define the tropical projective space to be the quotient 
\[\T\bP^{|V|-1} :=\left((\R_{\ge 0})^V\setminus \{0\}\right)/\R_+,\]
where $\R_+$ acts on $(\R_{\ge 0})^V$ diagonally by scalar multiplication. We write $[u_j:j\in V]$ for the homogeneous coordinates on $\T \bP^{|V|-1}$. This projective space is equipped with a natural volume form, which is given by the expression
\begin{equation} \label{eq:volume_form_TP}
 d\vol = \prod_{j\in V\setminus \{q\}} d\log \frac{u_j}{u_q}    
\end{equation}
for each choice of `inhomogeneous coordinates' which identify the complement of the hypersurface $\{u_q = 0\}$ with tropical affine space via the map
\begin{align*}
\T\bP^{|V|-1} \setminus \{u_q= 0\}& \overset{\cong}{\longrightarrow} (\R_{\ge 0})^{|V|-1}, \\ 
[u_j:j\in V] & \longmapsto (t_j = u_j/u_{q}:j\in V\setminus \{q\}).
\end{align*}
The key point is that the equality $d \log t = - d \log 1/t $ implies that the right-hand sides of Equation \eqref{eq:volume_form_TP} for two different affine charts agree on the overlap, yielding a volume form on $\T\bP^{|V|-1}  $.

\begin{lem} 
With respect to the volume form in Equation \eqref{eq:volume_form_TP}, we have: 
\[
\sum_{K\sqcup \{q\} = V} 
\int_{[0,\infty)^K}e^{-\sum_{k\in K} D_ks_k}\left(1+\sum_{k\in K}e^{-s_k}\right)^{-\sigma} ds^K 
= \int_{\T\bP^{|V|-1}} \frac{\prod_{j\in V} u_j^{D_j}}{(\sum_{j\in V} u_j)^\sigma} d\vol.
\]

\end{lem}
\begin{proof}
We begin by noting that $\prod_{j\in V} u_j^{D_j}/(\sum_{j\in V} u_j)^\sigma$ is a well-defined function on $\T\bP^{|V|-1}$ because the numerator and denominator are homogeneous functions of equal degree, and the denominator is non-vanishing.  Consider the subdivision $\T\bP^{|V|-1} = \bigcup_{q\in V} R_q$ with 
\[
R_q = \{[u_j:j\in V]\in \T\bP^{|V|-1} : u_q=\max(u_j : j\in V)\}.
\]
Then we have: 
\begin{align*} 
\int_{\T\bP^{|V|-1}} \frac{\prod_{j\in V} u_j^{D_j}}{(\sum_{j\in V} u_j)^\sigma} d\vol 
& = \sum_{q\in V}\int_{R_q} \frac{\prod_{j\in V} u_j^{D_j}}{(\sum_{j\in V} u_j)^\sigma} d\vol \\
&= \sum_{q\sqcup K = V} \int_{[0,1]^K} \frac{\prod_{k\in K} t_k^{D_k}}{(1+\sum_{k\in K} t_k)^\sigma} \prod_{k\in K}\frac{dt_k}{t_k} 
\end{align*} 
where, in the second line, we used the inhomogeneous coordinates $(t_k:k\in K)$ given by $t_k = u_k/u_q$. The conclusion follows by the change of variables $t_k = e^{-s_k}$. 
\end{proof}

We apply the above lemma to \eqref{eq:GXD}. We rewrite the integral over $\T\bP^{|V|-1}$ as an integral over the simplex $\nabla = \{\sum_{j\in V} u_j =1\}$, which is a slice of the diagonal action on $\R^V_{\geq 0} \setminus \{0\}$. 
Writing $\{D_0,\dots,D_m\}$ for $\{D_j : j\in V\}$, we find that the restriction of $d\vol$ to $\nabla$ is
\[\frac{du_1 \ldots du_m}{u_0 u_1 \ldots u_m}.\]
Thus we have  
\begin{align*} 
G_X(D) &= \frac{\prod_{j=0}^{m} D_j}{\sigma} 
\int_{\nabla} \prod_{j=0}^m u_j^{D_j-1} du_1\ldots du_m\\
& = \frac{\prod_{j=0}^{m} D_j}{\sigma} \cdot 
\frac{\Gamma(D_0) \cdots \Gamma(D_m)}{\Gamma(\sigma)} \\
& = \Gamma_X(D),
\end{align*} 
using a well-known integral due to Dirichlet \cite{Dirichlet1839},\footnote{This integral generalizes the Euler integral of the first kind which defines the Beta function (from which it can be proved by induction), and expresses the fact that the Dirichlet distribution is normalized.} together with the identity $z\Gamma(z) = \Gamma(1+z)$ and the formula \eqref{eq:1Gamma_X}. This essentially completes the proof of Proposition \ref{pro:G_Gamma} and hence of Theorem \ref{thm:gamma}, with the only missing step being the computation of the asymptotic expansion of $G_X$, which we now provide:
\begin{lem} 
\label{lem:asymp}
For a fixed $D \in (\R_+)^V$, we have the asymptotic expansion
\[
G_X(y D) \sim \left. \hG_X\right |_{D_j \rightarrow y D_j} \qquad \text{as $y  \to +0$}. 
\] 
\end{lem} 
\begin{proof} 
We fix $D\in (\R_+)^V$ throughout the proof. 
For a nonempty subset $J\subset V$, we set 
\begin{align*} 
g_J(X,y) & := h_J(X,y) \cdot \prod_{j\in J} X_j^{y D_j}, \\ 
h_J(X,y) & :=   
\frac{\sum_{K\subset J} (-1)^{|K|} 
\left( \left(1+\sum_{j\in K} X_j\right)^{-y \sigma} -1\right)}{(-y \sigma) \cdot \prod_{j\in J} X_j} 
\end{align*} 
where $X_j$ ($j\in V$) and $y$ are variables in $(0,1]$. Via the change of variables $X_j = e^{-s_j}$, we have  
\[
G_X(yD) = 1+ \sum_{q,J}(-yD_q) \prod_{j\in J}(-yD_j) 
\int_{[0,1]^J} g_J(X,y) \prod_{j\in J} dX_j  
\]
where the summation range is the same as before. Note that the sum over $q,J$ is finite. It suffices to show that we can exchange the Taylor expansion of $g_J(X,y)$ in $y$ with the integral over $[0,1]^J$ to get the asymptotic expansion. For this we use Taylor's theorem:
\[
g_J(X,y) = \sum_{a=0}^{m-1}\frac{1}{a!} 
(\partial_y^a g_J)(X,0) \cdot y^a 
+ \frac{1}{m!} (\partial_y^m g_J)(X,\xi(X)) \cdot y^m, \quad 
\exists \xi(X)\in [0,y].  
\]
Note that each term $(\partial_y^a g_J)(X,0)$ is a linear combination of products of the integrands defining the local integrals $I_{\ell;\vec{m}}$, and hence is integrable on $[0,1]^J$ for the same reason that the local integrals are, namely the exponential decay arising from the bound \eqref{eqn:boundint}.
It remains to show that $|(\partial_y^m g_J)(X,y)|$ is bounded by an integrable function of $X$ (on $[0,1]^J$) which is independent of $y\in [0,1]$. As in the proof of the bound \eqref{eqn:boundint},
we can see that $h_J(X,y)$ extends to a smooth (even analytic) 
function in a neighbourhood of $[0,1]^J \times [0,1]$; 
the numerator in the definition of $h_J$ vanishes along $X_j=0$ for each $j\in J$ and thus $h_J$ does not have poles along $X_j=0$. 
Thus there exist smooth functions $f_a(X,y)$ in a neighbourhood of $[0,1]^J\times [0,1]$ such that 
\[
\partial_y^m g_J(X,y) = 
\sum_{a=0}^m f_a(X,y) \cdot \left(\textstyle\sum_{j\in J} D_j \log X_j\right)^a 
\prod_{j\in J}X_j^{y D_j} 
\]
and therefore we find a constant $C>0$ with 
\[
\left | \partial_y^m g_J(X,y) \right |\le C 
\sum_{a=0}^m \left|\textstyle\sum_{j\in J} D_j \log X_j\right|^a 
\]
for $(X,y) \in (0,1]^J \times (0,1]$. The right-hand side is integrable on $[0,1]^J$ with respect to $X$, and the lemma follows. 
\end{proof}

\subsection{Examples of the local integrals} 

Recall that $\hGamma_X$ is expanded in the $\zeta$-values $\zeta(k)$ with $k\ge 2$ (see \eqref{eq:2Gamma_X}). The identity $\hG_X = \hGamma_X$ determines some of the local integrals $I_{\ell;\vec{m}}$ \eqref{eqn:Ilm} in terms of $\zeta(k)$. In general, however, the identity only shows that certain polynomial expressions in the local integrals equal $\zeta(k)$; it seems that  individual local integrals cannot necessarily be written as polynomials in $\zeta(k)$. 

A \emph{local integral of weight $k$} is a real number belonging to the set 
\[
\{I_{\ell;\vec{m}} : \ell + |\vec{m}| + \dim(\vec{m}) = k, \  \ell\ge 1, \ \dim(\vec{m})\ge 1\} 
\] 
where we set $\dim(\vec{m})=p$ and $|\vec{m}|=\sum_{i=1}^p m_i$ 
for $\vec{m} =(m_1,\dots,m_p) \in (\Z_{\ge 0})^p$. 
We can easily see that there are $\pi(k-1) + \pi(k-2) + \cdots +\pi(1)$ many local integrals of weight $k$, where $\pi(j)$ denotes the number of partitions of $j\in\N$. On the other hand, we obtain $\pi(k)-1$ relations\footnote{Note that $\pi(k)$ is the dimension of the space of symmetric functions of degree $k$. We have one relation fewer since the coefficients in front of $D_j^k$ of the degree-$k$ parts of both $\log \hG_X$ and $\log \hGamma_X$ vanish.} in weight $k$ from the identity $\hG_X = \hGamma_X$. Therefore, as $k$ grows, the number of local integrals becomes far greater than the number of relations among them. 

\begin{table}[h] 
\begin{center} 
\begin{tabular}{r|cccccccccccc}
weight  $k$                       & 2 & 3 & 4 & 5 & 6  & 7      & 8   & 9  & 10 & 11 & 12 & 13 \\  \hline 
$\#$ of local integrals & 1 & 3 & 6 & 11 & 18 & 29 & 44 & 66 & 96  & 138 & 194 & 271 \\
$\#$ of relations         & 1 & 2 & 4 & 6 & 10   & 14 & 21 & 29 & 41 & 55 & 76      & 100
\end{tabular}
\end{center} 
\end{table}  
The local integral $I_{1;m}$ of weight $m+2$ can be computed explicitly. 
\begin{align}
\label{eq:I1m}  
\begin{split} 
I_{1;m} & = -\int_0^\infty s^m \log(1+e^{-s}) ds \\ 
& = \sum_{n=1}^\infty\int_0^\infty s^m  \frac{(-1)^n}{n} e^{-ns} ds 
= m! \sum_{n=1}^\infty \frac{(-1)^n}{n^{m+2}} \\ 
&= - m! \left(1-\frac{1}{2^{m+1}}\right) \zeta(m+2). 
\end{split} 
\end{align} 

\begin{rmk} Viewing $\hG_X$, $\hGamma$ as functions of $D=(D_1,\dots,D_m)$ and writing $\hG_X= G_{X,m}(D_1,\dots,D_m)$, $\hGamma_X=\Gamma_{X,m}(D_1,\dots,D_m)$, we have 
\begin{align*} 
G_{X,m-1}(D_1,\dots,D_{m-1}) &= G_{X,m}(D_1,\dots,D_{m-1},0), \\ 
\Gamma_{X,m-1}(D_1,\dots,D_{m-1}) & =\Gamma_{X,m}(D_1,\dots,D_{m-1},0).
\end{align*} 
Therefore we can regard $\hG_X$, $\hGamma_X$ as symmetric functions in infinitely many variables $(D_1,D_2,D_3,\dots)$, and we obtain the maximal number of relations by doing so. 
\end{rmk}

\subsubsection*{In weight 2} We have only one local integral $I_{1;0} = -\frac{1}{2} \zeta(2)$. 

\subsubsection*{In weight 3} We have 3 local integrals $I_{2;0}$, $I_{1;0,0}$, $I_{1;1}$. The identity $\hG_X = \hGamma_X$ together with \eqref{eq:I1m} shows: 
\begin{align}
\label{eq:local_integral_wt_3}
\begin{split} 
I_{1;1} & = -\frac{3}{4} \zeta(3), \\  
I_{2;0} & = -\frac{1}{4} \zeta(3), \\ 
I_{1;0,0} & = -\frac{5}{12} \zeta(3).  
\end{split} 
\end{align} 

\subsubsection*{In weight 4} We have 6 local integrals $I_{3;0}$, $I_{2;1}$, $I_{2;0,0}$, $I_{1;2}$, $I_{1;1,0}$, $I_{1;0,0,0}$. We obtain $I_{1;2} = -\frac{7}{4} \zeta(4)$ from \eqref{eq:I1m} and the following four relations from $\hG_X = \hGamma_X$: 
\begin{align} 
\label{eq:relation_wt_4}
\begin{split} 
\frac{1}{2} I_{1;2} + \frac{1}{2} I_{2;1} + \frac{1}{3} I_{3;0} + \zeta(4) &= 0, \\
4 I_{1;0}^2 + 2 I_{1;2} + \zeta(4) & = 0, \\ 
\frac{1}{2} I_{1;2} + \frac{3}{2} I_{2;1} - 2 I_{1;1,0} - \frac{3}{2} I_{2;0,0} + \zeta(4) & = 0, \\ 
2 I_{1;2} - 8 I_{1;1,0} + 4 I_{1;0,0,0} + \zeta(4) & =0, 
\end{split} 
\end{align}
where the second equation reduces to the well-known identity $\zeta(4) = \frac{2}{5} \zeta(2)^2$. 

\subsubsection*{Other examples} 
Considering the case where $\{D_j : j \in V\} = \{D_1,D_2\}$ and comparing the coefficient of $D_1^{n-1} D_2$ of the identity $\hG_X= \hGamma_X$, we get a linear relation among the local integrals: 
\[
(n-1)! \zeta(n) + \sum_{i=1}^{n-2} \binom{n-1}{i} I_{n-1-i;i} + 2 I_{n-1,0} = 0, 
\] 
or equivalently, 
\begin{equation} 
\label{eq:zeta_n}
\zeta(n) = \frac{1}{(n-1)!} \int_{-\infty}^\infty \left(\log(1+e^s)\right)^{n-1} - (\max(0,s))^{n-1} ds. 
\end{equation} 
This generalizes the first equation of \eqref{eq:relation_wt_4}. 

There are other examples of the local integrals which can be expressed in terms of the $\zeta$-values: 
\begin{align*} 
I_{2;2} & = \frac{29}{8} \zeta(5) - 2 \zeta(2) \zeta(3), \\ 
I_{2;4} & = \frac{753}{8} \zeta(7) - 42 \zeta(3) \zeta(4) - 24 \zeta(2) \zeta(5). 
\end{align*} 

\subsection{Proof of \eqref{eqn:typeII}}
We use the above results for the local integrals to prove \eqref{eqn:typeII}. 

\begin{prop}
\label{prop:asymptotics_of_integral}
Let $V$ be a bounded domain in $\R^2$ containing the origin such that $\partial V$ is affine-linear in a neighbourhood of $\Sing(\max(0,b_1,b_2))$ and intersects it transversally. Let $L$ be the total affine length of the intersection $V \cap \Sing(\max(0,b_1,b_2))$. Then we have 
\begin{multline*} 
(-\log t)^3 \int_V \left(-\log_t(1+t^{-b_1}+t^{-b_2}) - \max(0,b_1,b_2)\right) db_1db_2\\
= \zeta(2) \cdot (-\log t) \cdot  L + \zeta(3) + \bigO 
\end{multline*} 
as $t\to +0$, for some $\epsilon>0$ depending only on $V$. 
\end{prop} 
\begin{proof} 
We decompose $V$ into small pieces $V_1,V_2,V_3,\dots$ and evaluate the integral locally. 
The integrand $-\log_t(1+t^{-b_1}+t^{-b_2}) -\max(0,b_1,b_2)$ is exponentially close to zero away from the tropical line $\Sing(\max(0,b_1,b_2))$ (see Figure \ref{fig:Maslov}). Hence, for any bounded domain $V_1$ with $\overline{V_1}\cap \Sing(0,b_1,b_2)=\emptyset$, we have 
\[
(-\log t)^3 \int_{V_1} (-\log_t(1+t^{-b_1}+t^{-b_2}) - \max(0,b_1,b_2)) db_1db_2 = O(t^\epsilon) 
\]
for some $\epsilon>0$ depending on $V_1$. Consider a domain $V_2\subset \R^2$ such that $\overline{V_2}$ intersects the tropical line $\Sing(\max(0,b_1,b_2))$ only along the edge $\{b_1=0, b_2<0\}$. We again assume that $\partial V_2$ is affine-linear in a neighbourhood of the edge and is transverse to it. Since the contribution (to the integral) away from the edge $\{b_1=0,b_2<0\}$ is exponentially small, we may assume that $V_2$ is of the form 
\[
-B<b_1<B, \quad f(b_1)<b_2<g(b_1) 
\]
for some negative affine-linear functions $f,g$. Because $t^{-b_2}$ is exponentially small in this region, we have 
\begin{align*}
(-\log t)^3&\int_{V_1}  \left(-\log_t(1+t^{-b_1} + t^{-b_2})-\max(0,b_1,b_2) \right) db_1db_2 \\
& = (-\log t)^3 \int_{-B}^B (-\log_t(1+t^{-b_1})-\max(0,b_1)) (g(b_1)-f(b_1))db_1+\bigO \\ 
& = (-\log t) \cdot L_2 \cdot \int_{ B\log t}^{-B \log t} \left( \log(1+e^x) - \max(0,x)\right) dx 
+ \bigO \\  
& = (-\log t) \cdot L_2 \cdot \zeta(2) + \bigO 
\end{align*} 
where $L_2 = g(0)-f(0)$ is the affine length of $V_2 \cap \Sing(\max(0,b_1,b_2))$ 
and we used \eqref{eq:zeta_n} in the last step. By the symmetry of the integrand under the affine transformation $(b_1,b_2)\to (-b_2,b_1-b_2)$, the same is true for regions intersecting the other edges. It now suffices to prove the statement for one particular $V$. Let $V$ be the rectangular region given by $-B<b_1<2B$, $-B<b_2<B$ with $B>0$ (see Figure \ref{fig:domain_decomp}). 
We decompose it into seven regions $W_1,\dots,W_7$ as shown. We have 
\begin{align*} 
&(-\log t)^3 \int_{W_1} \left(- \log_t(1+t^{-b_1}+t^{-b_2}) -\max(0,b_1,b_2)\right)db_1db_2  \\
& = 
\int_{[0,(-\log t)\cdot B]^2} \log(1+e^{-s_1} + e^{-s_2}) ds_1 ds_2  \\
& = \int_{[0,(-\log t)\cdot B]^2} \left(\log(1+e^{-s_1} + e^{-s_2}) 
- \log(1+e^{-s_1}) - \log(1+e^{-s_2})\right) ds_1 ds_2 \\ 
& \qquad + 2 (-\log t) \cdot B \cdot \int_{[0,(-\log t)\cdot B]} \log(1+e^{-s}) ds \\ 
& =I_{1;0,0} - 2 (-\log t) \cdot B \cdot  I_{1;0} + \bigO \\
& = -\frac{5}{12} \zeta(3) + (-\log t) \cdot B \cdot \zeta(2) + \bigO, 
\end{align*} 
where we used \eqref{eq:I1m} and \eqref{eq:local_integral_wt_3} in the last step. The integrals over the regions $W_2,W_3$ are the same by the affine symmetry. The integral over $W_4$ is given by
\begin{align*} 
(-\log t)^3  \int_{W_4} &\left(-\log_t(1+t^{-b_1}+t^{-b_2})-\max(0,b_1,b_2) \right) db_1db_2 \\
& = (-\log t)^3 \int_{0\le b_1 \le B, B\le b_1-b_2 \le B+b_1} \log_t(1+ t^{b_1} +t^{b_2-b_1}) db_1db_2 \\ 
& = \int_{[0,(-\log t)\cdot B]} x \log(1+e^{-x}) dx  + \bigO \qquad \text{($b_2$ integrated out)} \\ 
& = -I_{1;1} + \bigO = \frac{3}{4} \zeta(3) + \bigO,   
\end{align*} 
where we used \eqref{eq:I1m} in the last step. 
The integrals over $W_5,W_6$ are the same by the affine symmetry. The integral over $W_7$ is of order $\bigO$. The conclusion follows by summing up these contributions. 
\end{proof} 

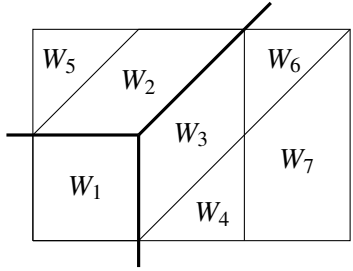
\begin{figure}
\centering 
\begin{tikzpicture}[x=1pt,y=1pt] 
\draw (110,10)--(110,50)--(110,90)--(230,90)--(230,10)--(110,10);
\draw (110,10)--(110,50)--(150,90)--(190,90)--(190,50)--(150,10)--(110,10);
\draw (230,90)--(190,50)--(190,10); 

\draw (130,30) node {$W_1$};  
\draw (150,70) node {$W_2$}; 
\draw (170,50) node {$W_3$};  
\draw (120,78) node {$W_5$}; 
\draw (178,20) node {$W_4$};  
\draw (205,78) node {$W_6$}; 
\draw (210,40) node {$W_7$}; 

\draw[very thick] (150,50)--(200,100); 
\draw[very thick] (150,50)--(100,50); 
\draw[very thick] (150,50)--(150,0);  

\end{tikzpicture} 
\caption{Decomposition of the domain $V$} 
\label{fig:domain_decomp} 
\end{figure} 

\begin{rmk}
If we only assume that $\partial V$ is \emph{smooth} (instead of linear) in a neighbourhood of the tropical line $\Sing(\max(0,b_1,b_2))$ and intersects it transversally, we get the same result except that the error term $\bigO$ in the right-hand side must be replaced with $O((-\log t)^{-1})$. This is because the 2-jet of $g-f$ (in the above proof) contributes to the term of order $(-\log t)^{-1}$.  More precisely, we are able to show the identity of distributions: 
\begin{multline*} 
-\log_t(1+t^{-b_1}+t^{-b_2}) = \max(0,b_1,b_2) \\ 
 - \frac{\zeta(2)}{ (-\log t)^{2}}  \delta_{\Sing(\max(0,b_1,b_2))} -\frac{\zeta(3)}{(-\log t)^{3}} \delta_{(0,0)} 
+ O\left(\frac{1}{(-\log t)^{4}}\right)  
\end{multline*} 
where $\delta_A$ means the delta measure supported on $A$ (with $A$ equipped with the affine measure).  
We plan to explore such distributions on tropical spaces in a future paper. 
\end{rmk}

\section{Periods of cycles that are mirror to line bundles}
\label{sec:phase} 

\subsection{Varying the phases}

Let $t>0$ be a small positive number and 
let $\theta = (\theta_q)_{q\in V}$ be a collection of real numbers. 
Recall that we define
\[
f_{t,\theta}(z) := \sum_{q\in V} e^{\iu \theta_q} t^{\lambda_q} z^q
\]
and
\[
\Zring_{t,\theta} := \{f_{t,\theta}(z) = 1\} \subset P_{\C^*}. 
\]
Let $C_{t,\theta}\subset \Zring_{t,\theta}$ be the parallel transport of the positive real cycle $C_t^+ \subset \Zring_t$ as we vary $\theta$ continuously from $\theta=0$. Let $\Omega_{t,\theta}$ be the holomorphic volume form on $\Zring_{t,\theta}$ defined by 
\[
\Omega_{t,\theta} = \left. \frac{d\log z_0 \wedge d\log z_1 \wedge 
\cdots \wedge d\log z_n}{df_{t,\theta}}\right|_{\Zring_{t,\theta}}  
\]
where $(z_0,\dots,z_n)$ are coordinates on $P_{\C^*} \cong (\C^*)^{n+1}$. 
In this section, we prove the following generalization of Theorem \ref{thm:periodR}: 
\begin{thm} 
\label{thm:cycle_with_phase}
We have
\[
\int_{C_{t,\theta}\subset Z_{t,\theta}} \Omega_{t,\theta} 
= \int_X t^{-\omega_\lambda} \cdot e^{-\sum_{q\in V} \iu \theta_q D_q} \cdot 
\hG_X + \bigO \quad \text{as $t\to +0$, for some $\epsilon>0$.}
\]
\end{thm} 

We can formally obtain this from Theorem \ref{thm:periodR} by substituting $\lambda_q+\iu(\theta_q/\log t)$ for $\lambda_q$. Note however that this does not logically follow from Theorem \ref{thm:periodR} since  in general `analytic continuation' does not commute with `asymptotic expansion'. We will justify this substitution via a tropical calculation similar to Section \ref{sec:periodR}. 

When $\theta_q = 2\pi \nu_q$, we have $C_t^{(\nu)} = C_{t,\theta}$. Therefore, Theorem \ref{thm:Bat_Gamma_lb} follows by combining Theorems \ref{thm:cycle_with_phase} and \ref{thm:gamma}.

\subsection{A tropical construction of the cycle $C_{t,\theta}$}
\label{sec:construction_Cttheta} 
We first construct a cycle approximately contained in $\Zring_{t,\theta}$ by ``sliding'' the positive real cycle $B_t=\Log_t(C_t^+)$ in the purely imaginary direction. Then we modify it to an actual cycle $C_{t,\theta}$ in $\Zring_{t,\theta}$. A similar construction for ``semi-tropical'' cycles appears in the work of Abouzaid \cite{Abouzaid:homogeneous, Abouzaid:HMS_toric}, Fang, Liu, Treumann and Zaslow \cite{FLTZ:T-dual}, and Hanlon \cite{Hanlon:monodromy}. 
 
Let $\kappa>0$ be a positive real number and 
let $N_\kappa(\Delta_\lambda)\subset P_\R$ denote the $\kappa$-neighbourhood of $\Delta_\lambda$
\[
N_\kappa(\Delta_\lambda) = \{ p\in P_\R: \beta_q(p) \ge -\kappa, \ \forall q\in V\}.  
\]
We choose $\kappa>0$ sufficiently small so that the following holds\footnote{\label{foot:correction} After the publication of the present paper, Yamamoto \cite[Remark 3.2]{Yamamoto:period_tropical} pointed out that we cannot choose $\kappa>0$ so that the condition (ii) holds and that we need to restrict to the $\kappa$-neighbourhood $N_\kappa(\partial \Delta_\lambda) = \{p \in P_\R : -\kappa\le \min_{q\in V} \beta_q(p)\le \kappa \}$ of the boundary of $\Delta_\lambda$ instead of $N_\kappa(\Delta_\lambda)$. (We should impose the condition (ii) only for $p\in N_\kappa(\partial \Delta_\lambda)$). Because of this, Lemma \ref{lem:grad} below is not correct as stated. Yamamoto \cite[Section 3]{Yamamoto:period_tropical} also provided detailed corrections. We thank Yamamoto for finding the errors and fixing them.}: 
\begin{itemize} 
\item[(i)]  for every subset $K\subset V$, if the regions $N_\kappa(\Delta_\lambda) \cap\{\beta_q \le \kappa\}$ with $q\in K$ intersect, then the facets $\Delta_\lambda\cap \{\beta_q=0\}$ with $q\in K$ intersect; in this case $K$ is linearly independent as a subset of $Q_\R$. 
\item[(ii)]  for every $p\in P_\R$, $K=\{k \in V: \beta_k(p) \le \min_{q\in V}(\beta_q(p))+\kappa\}$ is linearly independent as a subset of $Q_\R$. 
\end{itemize} 
We choose a smooth function $\phi\colon N_\kappa(\Delta_\lambda) \to P_\R$ such that the following holds for every $q\in V$: 
\begin{equation} 
\label{eq:bdry_cond_phi} 
\langle q,\phi(p)\rangle = - \theta_q \qquad 
\text{whenever} \quad \beta_q(p)\le \kappa. 
\end{equation} 
Such a function exists because of the above condition (i). 
Define a map $\Phi_t\colon N_\kappa(\Delta_\lambda) \to P_\C$ by 
\[
\Phi_t(p) = p + \iu \frac{\phi(p)}{\log t}.  
\]
We naturally extend the functions $\beta_q \colon P_\R \to \R$, 
$i_t \colon P_\R \to P_{\C^*}$ to the functions $\beta_q\colon P_\C\to \C$, $i_t \colon P_\C \to P_{\C^*}$ by the formulae:  
\[
\beta_q(p) = \langle q, p\rangle + \lambda_q, \qquad 
i_t(p_0,p_1,\dots,p_n) = (t^{p_0},t^{p_1},\dots,t^{p_n}). 
\]
The following is immediate from the definition. 
\begin{lem}
\label{lem:approx_cycle} 
For $p \in N_\kappa(\Delta_\lambda)$, 
we have that 
$|f_{t,\theta}(i_t(\Phi_t(p))) - f_t(i_t(p))| \le 2|V| t^\kappa$. 
\end{lem} 
\begin{proof} 
We have 
\begin{align*} 
f_t(i_t(p)) & = \sum_{q\in V} t^{\beta_q(p)} \\ 
f_{t,\theta}(i_t(\Phi_t(p))) &= \sum_{q\in V} e^{\iu(\theta_q+\langle q,\phi(p)\rangle)} t^{\beta_q(p)}. 
\end{align*} 
By the condition \eqref{eq:bdry_cond_phi}, the summands with $\beta_q(p) \le \kappa$ coincide. Since the summands with $\beta_q(p)>\kappa$ have norm bounded by $t^\kappa$, the conclusion follows. 
\end{proof} 
The lemma implies that the cycle $i_t(\Phi_t(B_t))$ is approximately contained in 
the hypersurface $\Zring_{t,\theta}$. 
We shall make it an actual cycle in $\Zring_{t,\theta}$ by a $C^1$-small perturbation. In the following, we will fix a Hermitian norm on $P_\C\cong \C^{n+1}$. 

\begin{prop}
\label{prop:delta} 
Set $R_t=\{p\in P_\R: \frac{1}{2} \le f_t(i_t(p))\le \frac{3}{2}\}$. 
For sufficiently small $t>0$, there exists a smooth function $\delta_t \colon R_t \to P_\C$ such that 
 \begin{itemize}
\item[(a)] we have $f_{t,\theta}(i_t(\Phi_t(p)+\delta_t(p)))= f_t(i_t(p))$ for $p\in R_t$, and   
\item[(b)] $\|\delta_t\|_{C^1}=O(t^\kappa)$, where 
$\|\cdot\|_{C^1}$ denotes the $C^1$-norm over the region $R_t$. 
\end{itemize} 
\end{prop} 

Using the function $\delta_t$ in this proposition, we define 
\[
C_{t,\theta} := i_t(\tPhi_t(B_t))
\]
where $\tPhi_t\colon R_t \to P_\C$ is defined by $\tPhi_t(p) = \Phi_t(p) + \delta_t(p)$. This is a cycle contained in $\Zring_{t,\theta}$ and is homeomorphic to a sphere. 

\begin{example} We give an example of the phase-shifting function $\phi\colon N_\kappa(\Delta_\lambda) \to P_\R$ in the case of $Y_{\Delta_\lambda} = \C\bP^2$. In Figure \ref{fig:phase}, we present $\phi$ as a vector field near the boundary of the moment polytope $\Delta_\lambda$; the boundary condition \eqref{eq:bdry_cond_phi} says that $p+\phi(p)$ with $p\in \partial \Delta_\lambda$ lies on the boundary of a shifted polytope (dotted line). Following the method in  \cite{FLTZ:T-dual}, we can give such $\phi$ via a piecewise linear function $f$ on the fan (drawn in thin lines) dual to the polytope $\Delta_\lambda$: $\phi$ can be given as the gradient of a smoothing of the piecewise-linear function $f$ that takes the values  $\theta_1$, $\theta_2$, $\theta_3$, respectively, at the primitive generators $(-1,0)$, $(0,-1)$, $(1,1)$ of the 1-dimensional cones. 
\begin{figure}[h]
\centering 
\begin{tikzpicture} 
\draw[very thick] (-1,-1) -- (2,-1) -- (-1,2) -- (-1,-1); 
\draw[very thin] (0,0)--(2,2); 
\draw[very thin] (-3,0) -- (0,0) -- (0,-2); 
\draw[dotted] (-1.5,-0.5) -- (1,-0.5) -- (-1.5,2)--(-1.5,-0.5); 

\draw[thin, ->] (-1,-0.4) -- (-1.5,0.1); 
\draw[thin, ->] (-1,-0.7) -- (-1.5, -0.2); 
\draw[thin, ->] (-1,-1) -- (-1.5,-0.5); 
\draw[thin, ->] (-0.6,-1) -- (-1.1,-0.5); 
\draw[thin, ->] (-0.2,-1) -- (-0.7,-0.5); 

\draw[thin, ->] (1,-1) -- (0,-0.5); 
\draw[thin,->] (1.5,-1) -- (0.5,-0.5); 
\draw[thin,->] (2,-1) -- (1,-0.5); 
\draw[thin, ->] (1.6,-0.6) -- (0.6,-0.1); 
\draw[thin, ->] (1.2,-0.2) -- (0.2,0.3); 

\draw[thin, ->] (0.2,0.8) -- (-0.3,0.8); 
\draw[thin, ->] (-0.1,1.1) -- (-0.6,1.1); 
\draw[thin, ->] (-0.4,1.4) -- (-0.9,1.4); 
\draw[thin,->] (-0.7,1.7) -- (-1.2,1.7); 
\draw[thin, ->] (-1,2) -- (-1.5,2); 
\draw[thin,->] (-1,1.65) -- (-1.5,1.65); 
\draw[thin, ->] (-1,1.3) -- (-1.5,1.3); 
\draw[thin, ->] (-1,0.9) -- (-1.5,0.9); 
\draw[thin, ->] (-1,0.5) -- (-1.5,0.5); 

\draw (-1.5,-1.5) node {\scriptsize $\phi = (-\theta_1,-\theta_2)$}; 
\draw (3,-0.2) node {\scriptsize $\phi = (\theta_2+\theta_3, -\theta_2)$}; 
\draw (-3,1) node {\scriptsize $\phi = (-\theta_1, \theta_1+\theta_3)$}; 

\end{tikzpicture}
\caption{A phase-shifting function $\phi$ as a vector field ($Y_{\Delta_\lambda}= \C\bP^2$).}
\label{fig:phase}
\end{figure}
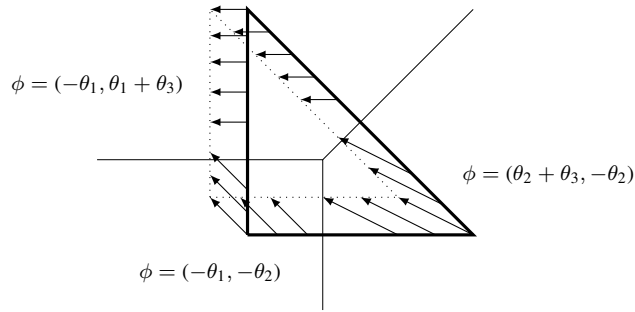
\end{example}

We end this subsection with the proof of Proposition \ref{prop:delta}. It is elementary and based on a standard method in tropical geometry. The uninterested reader can safely skip the proof since the details will not be used later. 
\begin{proof}[Proof of Proposition \ref{prop:delta}]
The region $R_t$ is contained in $N_\kappa(\Delta_\lambda)$ for sufficiently small $t>0$ because $f_t(i_t(p))\ge t^{-\kappa}$ for $p\in P_\R \setminus N_\kappa(\Delta_\lambda)$. In particular, $\Phi_t(p)$ is defined for $p\in R_t$. Define a holomorphic function $g_t\colon P_\C \to \C$ by $g_t(p) = f_{t,\theta}(i_t(p)) = \sum_{q\in V} e^{\iu\theta_q} t^{\beta_q(p)}$. By Lemma \ref{lem:approx_cycle}, $\Phi_t(p)$ with $p\in R_t$ is ``close'' to the fibre of the map $g_t$ at $f_t(i_t(p))$.  We will flow $\Phi_t(p)$ to a nearby point $\Phi_t(p)+\delta_t(p)$ lying in the fibre $g_t^{-1}(f_t(i_t(p)))$ by the gradient vector field of $g_t$. 
Let 
\[
\grad g_t = \left(\overline{\parfrac{g_t}{a_0}},\dots,
\overline{\parfrac{g_t}{a_n}}\right)
\]
be the gradient vector field of $g_t$ on $P_\C$ and define $\xi \colon R_t\to \C$ by 
\[
\xi_t(p) = f_t(i_t(p)) - g_t(\Phi_t(p)). 
\]
Lemma \ref{lem:approx_cycle} gives $|\xi_t(p)| \le 2 |V|t^\kappa$. 
Consider the differential equation for an unknown function $c(p;s) \colon R_t\times [0,1] \to P_\C$ 
\[
\frac{d}{ds} c(p;s) = \xi_t(p) \cdot \frac{\grad g_t}{|\grad g_t|^2}(c(p;s))
\]
with the initial condition $c(p;0) = \Phi_t(p)$. Suppose we have a global solution $c(p;s)$. Then the differential equation shows that $g_t(c(p;s))=g_t(\Phi_t(p))+s \xi_t(p)$ and therefore $c(p;1)$ lies in the fiber $g_t^{-1}(f_t(i_t(p)))$. Setting $c(p;1) = \Phi_t(p) + \delta_t(p)$, we have part (a) of the proposition. We will show that a solution exists and that the $C^1$-norm of $\delta_t$ is of order $t^\kappa$. 

For $p\in R_t$, let $s_p$ be the supremum of $\overline{s}\in [0,1]$ such that the flow $s\mapsto c(p;s)$ exists on the interval $[0,\overline{s})$. As discussed, $g_t(c(p;s))$ with $s\in [0,s_p)$ lies in the interval connecting $g_t(\Phi_t(p))$ and $f_t(i_t(p))$. Since $|f_t(i_t(p))-g_t(\Phi_t(p))| \le 2|V|t^\kappa$ by Lemma \ref{lem:approx_cycle} and $\frac{1}{2}\le f_t(i_t(p))\le \frac{3}{2}$, we may assume that $|g_t(c(p;s))|\ge \frac{1}{3}$ for sufficiently small $t>0$. Thus by Lemma \ref{lem:grad} below, we have 
\begin{equation} 
\label{eq:grad_g_estimate}
|(\grad g_t)(c(p;s))| \ge  (-\log t) \rho_0>0 
\end{equation} 
for some constant $\rho_0>0$ for a sufficiently small $t>0$. Therefore, the differential equation implies 
\[
\left| \frac{d}{ds} c(p;s) \right| = \frac{|\xi_t(p)|}{|(\grad g_t)(c(p;s))|} \le \frac{|\xi_t(p)|}{(-\log t) \rho_0} \le 2 \rho_0^{-1} |V| (-\log t)^{-1} t^\kappa  
\]
for $s\in [0,s_p)$. 
This implies that the limit $\lim_{s\to s_p-0}c(p;s)$ exists. If $s_p<1$, the solution can be extended to a larger interval, contradicting the assumption. Therefore we must have $s_p=1$ and the solution exists on the interval $[0,1]$. 

By integrating the above estimate, we get the bound (for $s\in [0,1]$): 
\begin{equation} 
\label{eq:estimate_cas}
|c(p;s)-c(p;0)| \le \rho_1 (-\log t)^{-1}t^\kappa   
\end{equation}
with $\rho_1 = 2 \rho_0^{-1} |V|$. 
This gives a $C^0$-bound for $\delta_t(p) =c(p;1)-c(p;0)$. 

To obtain a $C^1$-bound for $\delta_t$, we use the differential equation 
\begin{align} 
\label{eq:linear_DE}
\begin{split} 
\frac{d}{ds} \parfrac{c(p;s)}{p_i} & = \parfrac{\xi_t}{p_i}(p) \frac{\grad g_t}{|\grad g_t|^2}(c(p;s)) \\
& \quad + \xi_t(p) \left( F(c(p;s)) \parfrac{c(p;s)}{p_i} + 
G(c(p;s)) \parfrac{\overline{c(p;s)}}{p_i}\right) 
\end{split}
\end{align} 
where $F(c)$, $G(c)$ are square matrices of size $n+1$ (viewed as endomorphisms of $P_\C$) whose $(j,k)$-entries are given by 
\begin{align*} 
F(c)_{jk} & = -\frac{1}{|\grad g_t|^{4}} \overline{\parfrac{g_t}{c_j}}
\sum_{l=0}^n \overline{\parfrac{g_t}{c_l}} 
\parfrac{^2g_t}{c_l\partial c_k}  \\ 
G(c)_{jk} & = \frac{1}{|\grad g_t|^2} 
\overline{\parfrac{^2g_t}{c_j \partial c_k}}
- \frac{1}{|\grad g_t|^4} 
\overline{\parfrac{g_t}{c_j}} \sum_{l=0}^n \parfrac{g_t}{c_l}\overline{\parfrac{^2 g_t}{c_l \partial c_k}}. 
\end{align*} 
Given the function $c(p;s)$, \eqref{eq:linear_DE} can be viewed as a linear differential equation for $x(s)=\parfrac{c(p;s)}{a_i}$ as appears in Lemma \ref{lem:Gronwall} below. In view of Lemma \ref{lem:Gronwall}, it suffices to establish the following inequalities for sufficiently small $t>0$: 
\begin{align}
\label{eq:estimate_coeff} 
\begin{split} 
\max\left(
\left|\parfrac{\xi_t}{p_i}(p) \frac{\grad g_t}{|\grad g_t|^2}(c(p;s)) \right|, 
|\xi_t(p) F(c(p;s))_{jk}|,  
|\xi_t(p) G(c(p;s))_{jk}| \right) & \le \rho_2 t^\kappa \\
\left|\parfrac{c(p;0)}{p_i}\right| & \le \rho_2 
\end{split} 
\end{align} 
where $\rho_2>0$ is a constant independent of $p,s,t$. 
The second inequality follows from the initial condition
$c(p;0) = \Phi_t(p) = p + \iu \phi(p)/\log t$. 
Similarly to the proof of Lemma \ref{lem:approx_cycle}, we have 
\[
\left|\parfrac{\xi_t}{p_i}(p) \right| = (-\log t)
\left| \sum_{q\in V, \beta_q(p)>\kappa} q_i (1 - e^{\iu (\theta_q + \langle q, \phi(p)\rangle)}) t^{\beta_q(p)}  \right| 
\le \rho_3 (-\log t) t^\kappa 
\]
for some constant $\rho_3>0$. Therefore, using the estimate \eqref{eq:grad_g_estimate}, we get
\begin{equation} 
	\label{eq:estimate_derxi_grad} 
\left|\parfrac{\xi_t}{p_i}(p) \frac{\grad g_t}{|\grad g_t|^2}(c(p;s)) \right|
\le \rho_0^{-1} \rho_3 t^\kappa.  
\end{equation} 
Since $p\in R_t$, we have $f_t(i_t(p))\le \frac{3}{2}$ and hence 
$\beta_q(p) \ge \log_t (3/2)$. This together with the estimate \eqref{eq:estimate_cas} gives 
\begin{align*}
\Re (\beta_q(c(p;s))) & \ge \Re(\beta_q(c(p;0))) - |q| \rho_1 (-\log t)^{-1} t^\kappa \\
& \ge \log_t(3/2) - \rho_4 (-\log t)^{-1} t^\kappa. 
\end{align*} 
with $\rho_4 := \rho_1 \max_{q\in V}(|q|) $, 
where we used $\Re(\beta_q(c(p;0)))= \Re(\beta_q(\Phi(p)))=\beta_q(p)$. 
Therefore we have 
\begin{align*} 
\left| \parfrac{g_t}{c_j}(c(p;s)) \right | 
& = (-\log t) \left| \sum_{q\in V} e^{\iu\theta_q} q_j t^{\beta_q(c(p;s))} \right| \\
& \le (-\log t) \sum_{q\in V} |q_j| t^{\Re(\beta_q(c(p;s)))} \\ 
& \le (-\log t) \rho_5 \cdot t^{\log_t(3/2) - \rho_4 (-\log t)^{-1} t^\kappa} \le 
(-\log t) \rho_6 
\end{align*}
for some constant $\rho_5>0$ and $\rho_6 = \frac{3}{2} \rho_5 \exp(\rho_4)$. 
Similarly we have
\[
\left| \parfrac{^2 g_t}{c_j \partial c_k}(c(p;s))\right| 
\le (-\log t)^2 \rho_7 
\]
for some constant $\rho_7>0$. 
These estimates together with \eqref{eq:grad_g_estimate} and $|\xi_t(p)|\le 2|V| t^\kappa$ imply  
\begin{equation} 
	\label{eq:estimate_xiF_xiG} 
\max\left(|\xi_t(p) F(c(p;s))_{jk}|, |\xi_t(p) G(c(p;s))_{jk}|\right) \le \rho_8 t^\kappa 
\end{equation}
for some constant $\rho_8>0$. The estimates \eqref{eq:estimate_derxi_grad}, \eqref{eq:estimate_xiF_xiG} imply the first inequality of \eqref{eq:estimate_coeff}. The proposition is proved. 
\end{proof} 

\begin{rmk} By a similar argument, we can prove that the function $\delta_t$ constructed in the above proof is small in the $C^\infty$-topology, i.e.~$\|\delta_t\|_{C^m} =O(t^\kappa)$ for all $m\ge 0$, although we do not need this result. 
\end{rmk}

We end this subsection with the two lemmas\footnote
{Lemma \ref{lem:grad} is wrong because $g_t$ can have a critical point away from $g_t=0$. The estimate there is correct for $p$ in a neighbourhood of $\partial \Delta_\lambda$. See footnote \ref{foot:correction} and \cite[Lemma 3.6]{Yamamoto:period_tropical}.} used in the above proof. 

\begin{lem}
\label{lem:grad} 
Suppose $0<t<1$ and 
let $g_t\colon P_\C \to \C$ be the function defined by $g_t(p) = \sum_{q\in V} e^{\iu \theta_q} t^{\beta_q(p)}$. 
There exist  constants $\rho_1,\rho_2>0$ independent of $p\in P_\C$ such that
\[
\frac{|\grad g_t(p)|}{|g_t(p)|} \ge (-\log t)(\rho_1 - \rho_2 t^\kappa) 
\]
whenever $g_t(p) \neq 0$. In particular, the hypersurface $\Zring_{t,\theta}$ is smooth for sufficiently small $t>0$. 
\end{lem} 
\begin{proof} 
Fix $p\in P_\C$. Set $\beta_{q_0}(p) = \min_{q\in V}\beta_q(p)$ and $K=\{q\in V: \beta_k(p) \le \beta_{q_0}(p)+\kappa\}$. By the condition (ii) above, $K$ is linearly independent as a subset of $Q_\R$. We have 
\begin{align*} 
|\grad g_t(p)| &= (-\log t)\left| \sum_{q\in V} q e^{\iu \theta_q} t^{\beta_q(p)}\right | \\
& \ge (-\log t) \cdot |t^{\beta_{q_0}(p)}| \left( 
\left| \sum_{q\in K} q e^{\iu \theta_q} 
t^{\beta_q(p) - \beta_{q_0}(p)}\right| - \left|\sum_{q\in V\setminus K} 
q e^{\iu \theta_q} t^{\beta_q(p) - \beta_{q_0}(p)}\right|\right) \\
& \ge (-\log t) \cdot |t^{\beta_{q_0}(p)}| \left( \rho_1 \sum_{q\in K} |t^{\beta_q(p)-\beta_{q_0}(p)}|-\rho_3 t^\kappa\right) 
\end{align*} 
for some constants $\rho_1,\rho_3>0$, 
where we used the fact that $K$ is linearly independent and that all norms on a finite dimensional vector space are equivalent. 
Similarly we have 
\[
|g_t(p)| \le |t^{\beta_{q_0}(p)}| \left(\sum_{q\in K} |t^{\beta_q(p)-\beta_{q_0}(p)}|+\rho_4 t^\kappa \right)
\]
for some constant $\rho_4>0$. Combining these inequalities, we get 
\begin{align*} 
\frac{|\grad g_t(p)|}{|g_t(p)|} &\ge (-\log t)\frac{\rho_1-\rho_3 t^\kappa/\sum_{q\in K}|t^{\beta_q(p)-\beta_{q_0}(p)}|}{1+\rho_4 t^\kappa/\sum_{q\in K}|t^{\beta_q(p)-\beta_{q_0}(p)}|}\\
& \ge (-\log t) \left(\rho_1 - (\rho_1\rho_4+\rho_3) \frac{t^\kappa}{\sum_{q\in K}|t^{\beta_q(p)-\beta_{q_0}(p)}|}\right).  
\end{align*} 
Setting $\rho_2 = \rho_1 \rho_4 + \rho_3$, we obtain the conclusion since $\sum_{q\in K} |t^{\beta_q(p)-\beta_{q_0}(p)}| \ge 1$. 
\end{proof} 

\begin{lem}
\label{lem:Gronwall} 
Let $x(s)$ be a vector-valued function satisfying the differential equation 
\[
\frac{dx}{ds}(s) = y(s) + A(s) x(s) 
\]
where $y(s)$ is a vector-valued smooth function and $A(s)$ is a matrix-valued smooth function. 
Then we have 
\[
|x(1) - x(0)| \le \left (\int_0^1 |y(s) + A(s)x(0)| ds \right) e^{\int_0^1 |A(s)|ds}
\]
where $|A(s)|$ denotes the operator norm. 
\end{lem} 
\begin{proof} 
	We have the integral inequality: 
\begin{align*} 
|x(s)-x(0)| & = \left| \int_0^s y(u) + A(u) x(0) + A(u) (x(u)-x(0)) du\right| \\
& \le \int_0^1 |y(u)+A(u) x(0)| du + \int_0^s |A(u)| \cdot |x(u) - x(0)| du.  
\end{align*} 
The lemma follows by the Gronwall inequality. 
\end{proof} 

\subsection{Complex volume of a polytope}
We introduce a complex volume of an $m$-dimensional polytope `enclosed' by a complex hyperplane arrangement in $\C^m$.  The complex volume will appear in the calculation of periods of  $C_{t,\theta}$. 

Let $P\subset \R^m$ be a compact convex polytope (with non-empty interior) equipped with an orientation (as a manifold with corners). Let $\{F_i\}_{i\in I}$ be the set of facets (i.e.~faces of codimension one) of $P$. We assume that if $\bigcap_{i \in K} F_i \neq \emptyset$ for $K\subset I$, then $\bigcap_{i \in K} F_i$ is a face of codimension $|K|$. Suppose that we have a collection of  complex affine hyperplanes $\{H_i\}_{i \in I}$ in $\C^m$ labelled by the same index set $I$. We write $H_i=(\alpha_i+ \mu_i=0)$, where $\alpha_i\colon \C^m \to \C$ is a linear function and $\mu_i\in \C$ is a constant. We require that the hyperplane arrangement $\{H_i\}_{i \in I}$ satisfies the following (open) condition: for every subset $K\subset I$ such that $\bigcap_{i \in K} F_i \neq \emptyset$, the hyperplanes in the collection  $\{H_i\}_{i \in K}$ intersect transversally along a codimension $|K|$ affine subspace. 

Consider a smooth map 
\[
\Phi \colon P \to \C^m 
\]
satisfying 
\[
\Phi(F_i) \subset H_i 
\]
for all $i \in I$. 
Note that this condition (along with the requirement on $\{H_i\}_{i\in I}$ imposed above) determines the image of a vertex in $P$ under $\Phi$. Fix a holomorphic volume form $d\vol = r dc_1 \wedge \cdots \wedge dc_m$  on $\C^m$, for some $r \in \C^*$.  
The \emph{complex volume} of $(P,\Phi)$ is defined to be 
\begin{equation} 
	\label{eq:complex_volume} 
\vol_\C(P,\Phi) = \int_P \Phi^* (d\vol).  
\end{equation} 

\begin{lem} 
	\label{lem:complex_volume}
For a fixed volume form $d \vol$,    the complex volume $\vol_\C(P,\Phi)$ depends only on the hyperplane arrangement $\{H_i\}_{i \in I}$ on $\C^m$ indexed by facets of $P$. Moreover, it is a polynomial function of $\{\mu_i\}_{i\in I}$. 
\end{lem} 
\begin{proof} 
This follows from a repeated application of the Stokes theorem. More generally, for a polynomial $m$-form $\omega$ on $\C^m$, we can see that the integral 
\[
\int_P \Phi^*\omega 
\]
does not depend on the choice of $\Phi$ satisfying the above condition. Since any polynomial $m$-form on $\C^m$ is exact, by Stokes theorem, we can write it as the sum of integrals $\int_{F_i} (\Phi|_{F_i})^*\eta_i$ over $i \in I$ with $\eta_i$ being a polynomial $(m-1)$-form on $H_i$. The conclusion follows by induction on the dimension. 
\end{proof} 

\subsection{Computing periods of $C_{t,\theta}$ tropically}
By construction, the cycle $C_{t,\theta}$ is identified with the positive real cycle $C_t^+$ via the map $\tPhi_t=\Phi_t+\delta_t$. The tropical decomposition of $C_t^+$ in Section \ref{sec:decomposing-domain} then induces a decomposition of $C_{t,\theta}$. We can decompose and calculate the period of $C_{t,\theta}$ in almost the same way as before, with the only significant difference being that we use the complex volume (instead of the real volume) of a polytope and apply a complexified Duistermaat--Heckman theorem. We use the notation in Section \ref{sec:periodR}. 

Let $\kappa>0$ be as in Section \ref{sec:construction_Cttheta} and let $\epsilon>0$ be a sufficiently small number satisfying $0<\epsilon<\kappa/2$. Let $B_t^{q,K}\subset B_t$ be the region introduced in Section \ref{sec:decomposing-domain} defined by this $\epsilon$. We have 
\[
\int_{C_{t,\theta}} \Omega_{t,\theta} = \sum_{q,K} \int_{B^{q,K}_t} \tPhi_t^*i_t^* \Omega_{t,\theta} 
\]
where, as before, it suffices to consider the sum over all $q\in V$ and a subset $K\subset V$ not containing $q$ such that the facets of $\Delta_\lambda$ corresponding to $K\sqcup \{q\}$ have nonempty intersection. 

Let $(a,b_k,c_j)$ be affine coordinates on $P_\R$ from Section \ref{sec:approx_each_region} associated with the choice of $q$ and $K$. We naturally extend these coordinates to complex coordinates on $P_\C$. Similarly to equations \eqref{eq:f_t_on_region}--\eqref{eqn:negl}, we have 
\[
f_{t,\theta}\circ i_t(p) = e^{\iu\theta_q} t^a \left( 1+ \sum_{k\in K} e^{\iu(\theta_k-\theta_q)}t^{b_k} + h_t(p) \right) 
\]
where $h_t(p) = \sum_{m\in V\setminus (\{q\}\sqcup K)}e^{\iu(\theta_i - \theta_q)} t^{\beta_m(p) - a}$ satisfies the uniform estimate 
\[
h_t(p) \in O(t^\epsilon), \quad \parfrac{h_t}{p_i}(p) \in  O((-\log t) t^\epsilon)
\]
over $\tPhi_t(B^{q,K}_t)$; this follows from the $C^0$-estimate for $\delta_t$ in Proposition \ref{prop:delta}. Hence a calculation similar to equations  \eqref{eq:Omega_in_xy}--\eqref{eq:Omega_in_bc} shows 
\begin{align*} 
i_t^*\Omega_{t,\theta} & = \left. r_{q,K} (-\log t)^{n+1} \frac{da\wedge \bigwedge_{k\in K} db_k \wedge \bigwedge dc_j}{d(f_{t,\theta}\circ i_t)}\right|_{(f_{t,\theta} \circ i_t)^{-1}(1)} \\ 
& = (1+O(t^\epsilon))\cdot (-\log t)^n \cdot r_{q,K} \bigwedge_{k\in K} db_k \wedge \bigwedge _j dc_j
\end{align*} 
over $\tPhi_t(B^{q,K}_t)$, where $r_{q,K}>0$ is the number in Section \ref{sec:approx_each_region}. 

Consider the Riemannian metric $g$ on $B_t$ induced from the Euclidean metric on the ambient space $P_\R \cong \R^{n+1}$. Then the volume of $B_t$ with respect to $g$ is bounded as $t\to +0$; this follows from the leading asymptotics of periods $\int_{B_t} i_t^*\Omega_t$ in Theorem \ref{thm:periodR} and the estimate 
\begin{align*} 
\int_{B_t} d\vol_g & = \int_{B_t} |d(f_t \circ i_t)|\frac{dp_0 \cdots dp_n}{d(f_t \circ i_t)} \le \frac{\rho}{(-\log t)^n} \int_{B_t} i_t^*\Omega_t 
\end{align*} 
for some constants $\rho>0$, where we used the estimate $|d(f_t\circ i_t)| \le \rho (-\log t)$ over $B_t$. When we define the $C^1$-norm of $\delta_t|_{B_t}$ using the metric $g$, we have $\|\delta_t|_{B_t}\|_{C^1} \le \|\delta_t\|_{C^1}$; hence by Proposition \ref{prop:delta}, $\|\delta_t|_{B_t}\|_{C^1} = O(t^\kappa)$. Moreover it is easy to show that the $C^1$-norm of $\Phi_t$ (and hence of $\Phi_t|_{B_t}$) is bounded. From these facts, we obtain the approximation: 
\begin{multline*}
\int_{B^{q,K}_t}\tPhi_t^*i_t^*\Omega_{t,\theta} 
= (1+O(t^\epsilon)) \cdot (-\log t)^n \cdot r_{q,K}  
\int_{B_t^{q,K}} \tPhi_t^*\left(\textstyle{\bigwedge_{k\in K} db_k \wedge \bigwedge _j dc_j}\right) \\ 
= (1+O(t^\epsilon)) \cdot (-\log t)^n \cdot 
\left( r_{q,K} \int_{B_t^{q,K}} \Phi_t^*\left(\textstyle{ \bigwedge_{k\in K} db_k \wedge \bigwedge _j dc_j}\right) + O(t^\kappa) \right). 
\end{multline*} 
This corresponds to \eqref{eq:Omega_in_bc} in Section \ref{sec:approx_each_region}. 

Over $B^{q,K}_t$, we have $1= \sum_{k\in V} t^{\beta_k(p)} \le |V| t^a$, and thus $0\le a\le \log |V|/(-\log t)$. Therefore, by taking $t>0$ sufficiently small, we may assume that $0\le \beta_q=a \le \epsilon<\kappa$ and that $\beta_k = a+b_k \le 2 \epsilon < \kappa$ for $k\in K$ over $B^{q,K}_t$. This implies 
\[
\Phi_t^*b_k= b_k+ \iu \frac{\theta_k-\theta_q}{-\log t} \qquad \text{over $B^{q,K}_t$} 
\]
for $k\in K$ by the condition \eqref{eq:bdry_cond_phi}. Therefore we have 
\begin{align}
\label{eq:integral_vol_FqK(b)} 
\begin{split} 
r_{q,K} \int_{B_t^{q,K}} \Phi_t^*\left(\textstyle{ \bigwedge_{k\in K} db_k \wedge \bigwedge _j dc_j}\right) 
& = \int_{[0,\epsilon]^K} d\tilde{b} \int_{B_t^{q,K}\cap \{b=\tilde{b}\}} \Phi_t^*(d\vol_{q,K}) \\
& = \int_{[0,\epsilon]^K} d\tilde{b} \int_{F_{q,K}(\tilde{b})} (\Phi_t \circ s_{\tilde{b}})^*(d\vol_{q,K})
\end{split} 
\end{align}  
where $B_t^{q,K} \cap \{b=\tilde{b}\}$ denotes the subset of $B_t^{q,K}$ where the values of the coordinates $b=(b_k)_{k\in K}$ equal $\tilde{b}\in [0,\epsilon]^K$, $F_{q,K}(\tilde{b}) = \pi_c(B_t^{q,K} \cap \{b=\tilde{b}\})$ is its projection to the $c$-plane as appears in Section \ref{sec:approx_vol_polytope}, and $s_{\tilde{b}}$ denotes the (unique) section of the projection $\pi_c\colon B_t^{q,K} \cap \{b= \tilde{b}\} \to F_{q,K}(\tilde{b})$; we have $s_{\tilde{b}}(c)= (a_{q,K}(\tilde{b},c),\tilde{b},c)$. Also $d\vol_{q,K} = r_{q,K} \bigwedge_j dc_j$ (see \eqref{eq:residual_vol}) is now regarded as a holomorphic form on $P_\C$. 

Recall the function $a_{q,K}'(b)$ from \eqref{eq:approx_aqK}. It differs from $a_{q,K}(b,c)$ by a function of order $O(t^\epsilon)$. Moreover, we have the uniform estimate over the region $\pi_{b,c}(B^{q,K}_t)$ 
\[
\parfrac{a_{q,K}(b,c)}{c_j} = O(t^\epsilon). 
\]
In fact, by differentiating \eqref{eq:defining_eqn_Bt} with respect to $c_j$, we get 
\begin{multline*}
\parfrac{a_{q,K}}{c_j}(b,c) \left( 1+\sum_{m\notin \{q\}\sqcup K}\left(\partial_a \beta_m-1\right)t^{\beta_m(a_{q,K}(b,c),b,c)} \right) \\ 
= - \sum_{m\notin\{q\}\sqcup K} (\partial_{c_j}\beta_m) t^{\beta_m(a_{q,K}(b,c),b,c)}
\end{multline*}
where we regard $\beta_m=\beta_m(a,b,c)$ as an affine linear form in $a,b,c$; the above estimate follows by $\beta_m(a_{q,K}(b,c),b,c)\ge a_{q,K}(b,c)+\epsilon>\epsilon$ over $\pi_{b,c}(B^{q,K}_t)$. By this estimate, we can replace the section $s_b$ in \eqref{eq:integral_vol_FqK(b)} with $s_b'(c) = (a_{q,K}'(b),b,c)$ with error terms of order $O(t^\epsilon)$. Furthermore, by the same argument as in Section \ref{sec:approx_vol_polytope}, we can replace $F_{q,K}(b)$ with the polytope $F_{q,K}'(a'_{q,K}(b),b)$ with error terms of order $O(t^\epsilon)$ again. Namely, we have 
\begin{align*}
\int_{F_{q,K}(b)} (\Phi_t\circ s_b)^*d\vol_{q,K} & = \int_{F_{q,K}(b)} (\Phi_t\circ s_b')^*(d\vol_{q,K}) + O(t^\epsilon) \\
& = \int_{F_{q,K}'(a_{q,K}'(b),b)} (\Phi_t\circ s_b')^*(d\vol_{q,K}) + O(t^\epsilon) \\ 
& = \vol_\C\left(F_{q,K}'(a_{q,K}'(b),b), \Phi_t\right). 
\end{align*} 
In the last line, we identified the polytope $F_{q,K}'(a_{q,K}'(b),b)$ with its image by $s_b'$ (as we did in Section \ref{sec:approx_vol_polytope}) and considered its complex volume \eqref{eq:complex_volume} with respect to $d\vol_{q,K}$. Note that $\Phi_t(F_{q,K}'(a_{q,K}'(b),b))$ is contained in the affine subspace of $P_\C$ where the values of the complex affine-linear functions $a$, $b_k$ are fixed (again by the condition \eqref{eq:bdry_cond_phi}). 

Applying the inclusion-exclusion principle for the complex volumes of polytopes as we did in Section \ref{sec:approx_vol_polytope}, we arrive at the formula\footnote{We express the complex volume of  $F_{q,K}'(a,b)$ as a signed sum of the complex volumes of $\bigcup_{b'\in [0,\epsilon]^I} E_{q,K\sqcup I}(a,b,b')$. Then we use the fact that $\Phi_t^*(b'_i)$ (with $i\in I$) on $\bigcup_{b'\in [0,\epsilon]^I} E_{q,K\sqcup I}(a,b,b')$ equals $b'_i+\iu (\phi_i-\phi_q)/(-\log t)$ for  $a=a'_{q,K}(b)$ to factor out $db'$.}: 
\begin{multline*} 
\int_{C_{t,\theta}} \Omega_{t,\theta} 
= (1+O(t^\epsilon)) \ (-\log t)^n \\
\times \left( \sum_{q,K\subset J,q\notin J} (-1)^{|J\setminus K|}\int_{[0,\epsilon]^J}  \vol_\C\left(E_{q,J}(a_{q,K}'(b),b,b'),\Phi_t\right) db db' + O(t^\epsilon)\right) 
\end{multline*} 
where we write $b=(b_j)_{j\in K}$ and $b'=(b'_j)_{j\in J\setminus K}$ and used the holomorphic form $d \vol_{q,J}$ to define the complex volume for $(E_{q,J}(a_{q,K}'(b),b,b'),\Phi_t)$. 
This generalizes \eqref{eqn:perH}. 

Finally we apply a complex version of the Duistermaat--Heckman theorem. Consider the polytope $E_{q,J}(a,b)$ with $0<a, b_j<\epsilon$.  The image of $E_{q,J}(a,b)$ under $\Phi_t$ is contained in the complex affine subspace of $P_\C$ defined by 
\begin{align*}
\beta_q & = a + \iu \frac{\theta_q}{-\log t}, \\ 
\beta_j & = a+b_j+ \iu\frac{\theta_j}{-\log t} \qquad j\in J. 
\end{align*} 
by \eqref{eq:bdry_cond_phi}. 
The facets of $E_{q,J}(a,b)$ are given by $E_{q,J}(a,b) \cap \{\beta_m = a\}$ for $m\notin \{q\} \sqcup J$. These facets map (under $\Phi_t$) to the complex affine hyperplane given by 
\[
\beta_m-\beta_q= \iu \frac{\phi_m-\phi_q}{-\log t} 
\]
by \eqref{eq:bdry_cond_phi}. Therefore its complex volume $\vol_\C(E_{q,K}(a,b),\Phi_t)$ is a polynomial function of the constant terms of these complex affine-linear forms by Lemma \ref{lem:complex_volume}. By analytic  continuation of Lemma \ref{lem:DH}, we get 
\[
\vol_\C(E_{q,J}(a,b),\Phi_t) = \int_{Y_{\Delta_\lambda}}  
\exp\left(\omega_{\lambda}+\iu \frac{\sum_{k\in V} \theta_k D_k}{\log t} - \sum_{j\in J} b_j D_j - a \sigma\right) \cdot D_q \cdot \prod_{j\in J} D_j. 
\] 
We obtain this from Lemma \ref{lem:DH} by substituting $\lambda_k + \iu (\theta_k/\log t)$ for $\lambda_k$ for all $k\in V$. 

The rest of the argument works in exactly the same way as before, just by replacing $\lambda_k$ with $\lambda_k+\iu(\theta_k/\log t)$, and we arrive at Theorem \ref{thm:cycle_with_phase}. 

\bibliographystyle{gtart}
\bibliography{GammaSYZ}

\end{document}